\documentclass[11pt,a4paper,dvipsnames]{article}
\usepackage{geometry}
\usepackage[T1]{fontenc}
\usepackage[utf8]{inputenc}
\usepackage{wrapfig}
\usepackage{amsthm,lmodern,mathrsfs,amsmath,amssymb,mathtools,cases,tikz,enumerate,algorithm,algpseudocode,enumitem,comment,pgfplotstable,stmaryrd,bbm} 
\usepackage[multiple]{footmisc}
\usepackage{float}
\pgfplotsset{compat=1.7}
\usepgfplotslibrary{fillbetween}
\usetikzlibrary{calc}
\usetikzlibrary{decorations.pathreplacing}
\usetikzlibrary{shapes,intersections}
\usepgfplotslibrary{colormaps}
\usetikzlibrary{decorations.markings,arrows}
\usetikzlibrary{backgrounds}
\usetikzlibrary{fadings}
\tikzfading[name=fade right,
    left color=transparent!0,
    right color=transparent!100]
\usepackage{hyperref}

\usepackage[
    noabbrev,   
    capitalise, 
    nameinlink,  
]{cleveref} 

\usepackage[immediate]{silence} 
\usepackage{sectsty}
\DeactivateWarningFilters[temp]

\usepackage[labelfont={bf,sf}]{caption}
\usepackage{pgfplots}
\usepgfplotslibrary{colormaps}
\usepackage[toc,page]{appendix}
\usepackage[%
maxnames=20,
giveninits=true,
isbn=false,
sorting=none,
]{biblatex}
\newcommand\norm[1]{\left\lVert#1\right\rVert}

\newcommand{\N}{\mathbb{N}}

\newcommand{\Sym}{\mathbb{S}}
\newcommand{\Hil}{\mathcal{H}}
\newcommand\Fix{\operatorname{Fix}}

\newcommand{\prox}{{\rm{prox}}}

\DeclareMathOperator*{\minimize}{minimize}

\DeclareMathOperator*{\Argmin}{Argmin}
\DeclareMathOperator*{\gra}{gra}
\DeclareMathOperator*{\ran}{ran}
\DeclareMathOperator*{\zer}{zer}
\DeclareMathOperator*{\dom}{dom}
\DeclareMathOperator*{\Id}{Id}

\DeclareMathOperator*{\relint}{relint}

\DeclareMathOperator{\dist}{dist}

\newcommand*{\R}{\mathbb{R}}
\newcommand*{\weak}{\rightharpoonup}

\newenvironment{absolutelynopagebreak}
  {\par\nobreak\vfil\penalty0\vfilneg
   \vtop\bgroup}
  {\par\xdef\tpd{\the\prevdepth}\egroup
   \prevdepth=\tpd}

\AtBeginEnvironment{appendices}{%
  \crefalias{section}{appendix}%
}
\allsectionsfont{\normalfont\sffamily\bfseries}

\makeatletter
\@ifundefined{theHALG@line}{
    \newcommand{\theHALG@line}{\thealgorithm.\arabic{ALG@line}}
}{
    \renewcommand{\theHALG@line}{\thealgorithm.\arabic{ALG@line}}
}

\def\th@plain{
  \thm@headfont{\normalfont\sffamily\bfseries}
  \itshape 
}

\def\th@definition{
  \thm@headfont{\normalfont\sffamily\bfseries}
  \thm@notefont{\normalfont\sffamily\bfseries}
}

\newtheoremstyle{myStyle1}
  {0.3cm}
  {0.3cm}
  {\itshape}
  {}
  {\normalfont\sffamily\bfseries}
  {:}
  {.5em}
  {}

\makeatother

\newtheoremstyle{myStyle2}
  {0.3cm}
  {0.3cm}
  {}
  {}
  {\normalfont\sffamily\bfseries}
  {:}
  {.5em}
  {}

\makeatother

\theoremstyle{myStyle1}
\newtheorem{theorem}{Theorem}[section]
\newtheorem{proposition}{Proposition}[section]
\newtheorem{corollary}{Corollary}[section]
\newtheorem{assumption}{Assumption}[section]
\newtheorem{lemma}{Lemma}[section]

\newtheorem{definition}{Definition}[section]
\newtheorem{conjecture}{Conjecture}[section]

\theoremstyle{myStyle2}
\newtheorem{remark}{Remark}[section]

\makeatletter
\def\maketag@@@#1{\hbox{\m@th\normalfont\normalsize#1}}

\makeatother
\newcommand{\itemcref}[2]{\hyperref[#2]{\cref*{#1}\labelcref*{#2}}}

\crefname{equation}{}{}
\Crefname{equation}{}{}

\xdef\zero{0}
\xdef\one{1}

    \newcommand{\unitcirc}[3]{

  \xdef\markerlabel{#1}
  \xdef\startlabel{#2}
  \xdef\originlabel{#3}
  
  \marker{(0,0)}{0.08}
  \marker{(1,0)}{0.08}
  \draw[very thick] (0,0) circle (1);

  \ifx\markerlabel\one
  \node[right] at (1,0) {\footnotesize \startlabel};
  \node[left] at (0,0) {\footnotesize \originlabel};

  \fi
  }

\newcommand{\marker}[2]
{
  \pgfgettransformentries{\myxscale}{\@tempa}{\@tempa}{\myyscale}{\@tempa}{\@tempa}
  \draw[thick] ($#1+#2*(1/\myxscale,1/\myyscale)$)--($#1-#2*(1/\myxscale,1/\myyscale)$);
  \draw[thick] ($#1+#2*(-1/\myxscale,1/\myyscale)$)--($#1-#2*(-1/\myxscale,1/\myyscale)$);
}

\newcommand{\composition}[9]{

 \xdef\thetavar{#1}
\xdef\alphaone{#2}
\xdef\betaone{#3}
\xdef\alphatwo{#4}
\xdef\betatwo{#5}
\xdef\drawcircle{#6}
\xdef\writeprop{#7}
\xdef\fillgray{#8}
\xdef\startlabel{#9}

\ifx\fillgray\zero
\xdef\nbrinnercirc{20} 
\begin{scope}[shift={({1-\thetavar},0)},scale={\thetavar}]
\foreach \s in {0,...,\nbrinnercirc}
         {
           \pgfmathsetmacro{\circx}{{\alphaone-\betaone*cos(\s*360/20)}}
           \pgfmathsetmacro{\circy}{{\betaone*sin(\s*360/20)}}
           \pgfmathsetmacro{\distc}{{sqrt((\circx)^2+(\circy)^2)}}
           \draw[dashed] ({\alphatwo*\circx},{\alphatwo*\circy}) circle ({\distc*\betatwo});
           \marker{(\circx,\circy)}{0.08};
         }
       \end{scope}
       \fi

\ifx\fillgray\one
\xdef\nbrinnercircquad{15} 
 \begin{scope}[shift={({1-\thetavar},0)},scale={\thetavar}]
\foreach \s in {0,...,\nbrinnercircquad}
         {
           \pgfmathsetmacro{\circx}{{\alphaone-\betaone*cos(\s*90/\nbrinnercircquad)}}
           \pgfmathsetmacro{\circy}{{(\betaone*sin(\s*90/\nbrinnercircquad))}}
           \pgfmathsetmacro{\distc}{{sqrt((\circx)^2+(\circy)^2)}}
           \draw[name path=circleone] ({\alphatwo*\circx+\distc*\betatwo/2},{\alphatwo*\circy+\distc*\betatwo/2})--({\alphatwo*\circx-\distc*\betatwo/2},{\alphatwo*\circy-\distc*\betatwo/2});
           \draw[color=gray!25!white,fill=gray!25!white] ({\alphatwo*\circx},{\alphatwo*\circy}) circle ({\distc*\betatwo});
           
           \pgfmathsetmacro{\circx}{{\alphaone-\betaone*cos((\s+\nbrinnercircquad)*90/\nbrinnercircquad)}}
           \pgfmathsetmacro{\circy}{{(\betaone*sin((\s+\nbrinnercircquad)*90/\nbrinnercircquad))}}
           \pgfmathsetmacro{\distc}{{sqrt((\circx)^2+(\circy)^2)}}
           \draw[name path=circletwo] ({\alphatwo*\circx+\distc*\betatwo/2},{\alphatwo*\circy-\distc*\betatwo/2})--({\alphatwo*\circx-\distc*\betatwo/2},{\alphatwo*\circy+\distc*\betatwo/2});
           \draw[color=gray!25!white,fill=gray!25!white] ({\alphatwo*\circx},{\alphatwo*\circy}) circle ({\distc*\betatwo});

           \pgfmathsetmacro{\circx}{{\alphaone-\betaone*cos((\s+2*\nbrinnercircquad)*90/\nbrinnercircquad)}}
           \pgfmathsetmacro{\circy}{{(\betaone*sin((\s+2*\nbrinnercircquad)*90/\nbrinnercircquad))}}
           \pgfmathsetmacro{\distc}{{sqrt((\circx)^2+(\circy)^2)}}
           \draw[name path=circlethree] ({\alphatwo*\circx+\distc*\betatwo/2},{\alphatwo*\circy+\distc*\betatwo/2})--({\alphatwo*\circx-\distc*\betatwo/2},{\alphatwo*\circy-\distc*\betatwo/2});
           \draw[color=gray!25!white,fill=gray!25!white] ({\alphatwo*\circx},{\alphatwo*\circy}) circle ({\distc*\betatwo});
           
           \pgfmathsetmacro{\circx}{{\alphaone-\betaone*cos((\s+3*\nbrinnercircquad)*90/\nbrinnercircquad)}}
           \pgfmathsetmacro{\circy}{{(\betaone*sin((\s+3*\nbrinnercircquad)*90/\nbrinnercircquad))}}
           \pgfmathsetmacro{\distc}{{sqrt((\circx)^2+(\circy)^2)}}
           \draw[name path=circlefour] ({\alphatwo*\circx+\distc*\betatwo/2},{\alphatwo*\circy-\distc*\betatwo/2})--({\alphatwo*\circx-\distc*\betatwo/2},{\alphatwo*\circy+\distc*\betatwo/2});
           \draw[color=gray!25!white,fill=gray!25!white] ({\alphatwo*\circx},{\alphatwo*\circy}) circle ({\distc*\betatwo});
           
           \tikzfillbetween[of=circletwo and circlefour] {fill=gray!25!white};
           \tikzfillbetween[of=circleone and circlethree] {fill=gray!25!white};
         }

 \end{scope}
\fi

\ifx\fillgray\zero
\begin{scope}[shift={({1-\thetavar},0)},scale={\thetavar}]
  \draw (\alphaone,0) circle (\betaone);
  \end{scope}
\fi
\unitcirc{1}{\startlabel}{\originlabel}
 \ifx\writeprop\one
 \pgfmathsetmacro{\deltaone}{{(\alphaone)/(1-\alphaone)*(1-((1-\alphatwo)^2-(\betatwo)^2)/(1-\alphatwo))}}
 \pgfmathsetmacro{\deltatwo}{{(\alphatwo)/(1-\alphatwo)}}
 \pgfmathsetmacro{\deltathree}{{1-((((1-\alphaone)^2-(\betaone)^2)/(1-\alphaone))*(1-(((1-\alphatwo)^2-(\betatwo)^2)/(1-\alphatwo)))+((1-\alphatwo)^2-(\betatwo)^2)/(1-\alphatwo))}}
 \pgfmathsetmacro{\deltafour}{{(\deltaone*\deltatwo)/(\deltaone+\deltatwo)}}
 \pgfmathsetmacro{\alphacomp}{{(1-\thetavar)+\thetavar*\deltafour/(1+\deltafour)}}
 \pgfmathsetmacro{\betacomp}{{\thetavar*sqrt(\deltathree-\deltafour+\deltathree*\deltafour)/(1+\deltafour)}}
 \pgfmathsetmacro{\alphacomptwo}{{(1-\thetavar)+\thetavar*(\alphaone+\betaone)*(\alphatwo+\betatwo)*(1-((\betaone*\alphatwo)+(\betatwo*\alphaone))/(\alphaone*\alphatwo+\alphaone*\betatwo+\alphatwo*\betaone))}}
 \pgfmathsetmacro{\betacomptwo}{{\thetavar*(\alphaone+\betaone)*(\alphatwo+\betatwo)*(((\betaone*\alphatwo)+(\betatwo*\alphaone))/(\alphaone*\alphatwo+\alphaone*\betatwo+\alphatwo*\betaone))}}
  \pgfmathsetmacro{\alphacompthree}{{(1-\thetavar)+\thetavar*(abs(\alphaone)+\betaone)*(abs(\alphatwo)+\betatwo)*(1-((\betaone*abs(\alphatwo))+(\betatwo*abs(\alphaone)))/(abs(\alphaone)*abs(\alphatwo)+abs(\alphaone)*\betatwo+abs(\alphatwo)*\betaone))}}
 \pgfmathsetmacro{\betacompthree}{{\thetavar*(abs(\alphaone)+\betaone)*(abs(\alphatwo)+\betatwo)*(((\betaone*abs(\alphatwo))+(\betatwo*abs(\alphaone)))/(abs(\alphaone)*abs(\alphatwo)+abs(\alphaone)*\betatwo+abs(\alphatwo)*\betaone))}}
\fi

 \ifx\drawcircle\one 
 \draw[thick] (\alphacomptwo,0) circle (\betacomptwo);
 \fi
 
 \ifx\writeprop\one
 \node at (-2,-1) {\deltaone};
 \node at (-2,-0.5) {\deltatwo};
 \node at (-2,0) {\deltathree};
 \node at (-2,0.5) {\deltafour};
 \node at (-2,1) {\alphacomp};
 \node at (-2,1.5) {\betacomp};
 \node at (-1,1) {\alphacomptwo};
 \node at (-1,1.5) {\betacomptwo};
 \fi
 
}

\newcommand{\oneRiteration}[6]{%
  \pgfmathsetmacro{\cpar}{#1}
  \pgfmathsetmacro{\dpar}{#2}
  \pgfmathsetmacro{\al}{#3}
  \pgfmathsetmacro{\be}{#4}
  \pgfmathsetmacro{\th}{#5}
  \pgfmathsetmacro{\cTwo}{\cpar*\cpar}
  \pgfmathsetmacro{\dTwo}{\dpar*\dpar}
  \pgfmathsetmacro{\denC}{\cTwo+1}
  \pgfmathsetmacro{\denD}{\dTwo+1}
  \pgfmathsetmacro{\lam}{\be/\al}             
  \pgfmathsetmacro{\mu}{\al/\be}              
  \pgfmathsetmacro{\w}{\th*\be/(\al+\be)}     

  \def\zx{1}\def\zy{0}

  \pgfmathsetmacro{\xone}{\cTwo/\denC}
  \pgfmathsetmacro{\yone}{\cpar/\denC}

  \pgfmathsetmacro{\ronex}{(\cTwo-\lam)/\denC}
  \pgfmathsetmacro{\roney}{(1+\lam)*\cpar/\denC}

  \pgfmathsetmacro{\pdx}{(\dTwo*\ronex + \dpar*\roney)/\denD}
  \pgfmathsetmacro{\pdy}{(\dpar*\ronex + \roney)/\denD}

  \pgfmathsetmacro{\rtwox}{(1+\mu)*\pdx - \mu*\ronex}
  \pgfmathsetmacro{\rtwoy}{(1+\mu)*\pdy - \mu*\roney}

  \pgfmathsetmacro{\zpx}{(1-\w) + \w*\rtwox}
  \pgfmathsetmacro{\zpy}{\w*\rtwoy}

  \ifnum#6=1
      \ifdim \cpar pt > 1pt
        \draw[very thin] (-2,{-2/\cpar})--(2,{2/\cpar}) node[xshift=-4pt,yshift=5pt] {\footnotesize{$C_1$}};
    \else
        \draw[very thin] ({-2*\cpar},-2)--({2*\cpar},2) node[xshift=7pt,yshift=-4pt] {\footnotesize{$C_1$}};
    \fi   
          \ifdim \dpar pt > 1pt
        \draw[very thin] (-2,{-2/\dpar})--(2,{2/\dpar}) node[xshift=-4pt,yshift=5pt] {\footnotesize{$C_2$}};
    \else
        \draw[very thin] ({-2*\dpar},-2)--({2*\dpar},2) node[xshift=7pt,yshift=-4pt] {\footnotesize{$C_2$}};
    \fi   
  \fi

  \draw[thin] (\zx,\zy)--(\ronex,\roney);         
  \draw[thin] (\ronex,\roney)--(\rtwox,\rtwoy);   
  \draw[thin] (\rtwox,\rtwoy)--(\zpx,\zpy);    

  \marker{(\zx,\zy)}{0.08}
  \node[anchor=west] at (\zx,\zy) {\footnotesize{$z$}};

  \marker{(0,0)}{0.08}
  \node[anchor=east] at (0,0) {\footnotesize{$z^\star$}};

  \marker{(\ronex,\roney)}{0.08}
  \node[xshift=-6pt,yshift=6pt] at (\ronex,\roney) {\footnotesize{$r_1$}};

  \marker{(\rtwox,\rtwoy)}{0.08}
  \node[xshift=3pt,yshift=6pt] at (\rtwox,\rtwoy) {\footnotesize{$r_2$}};

  \marker{(\zpx,\zpy)}{0.08}
  \node[xshift=9pt,yshift=2pt] at (\zpx,\zpy) {\footnotesize{$z^+$}};
}
\begin{document}
\title{\Large \sffamily\bfseries Extending Douglas--Rachford Splitting for Convex Optimization}

\author{%
  Max Nilsson$^{\dagger}$, Anton Åkerman$^{\dagger}$, and Pontus Giselsson\\[0.5ex]
  {\small Department of Automatic Control, Lund University, Lund, Sweden}\\
  {\small \{\href{mailto:max.nilsson@control.lth.se}{max.nilsson},
  \href{mailto:anton.akerman@control.lth.se}{anton.akerman},
  \href{mailto:pontus.giselsson@control.lth.se}{pontus.giselsson}\}@control.lth.se}\\[1ex]
  {\small $^{\dagger}$These authors contributed equally to this work.}
}
\date{}
\maketitle
\begin{abstract}
    The Douglas–Rachford splitting method is a classical and widely used algorithm for solving monotone inclusions involving the sum of two maximally monotone operators. It was recently shown to be the unique frugal, no-lifting resolvent-splitting method that is unconditionally convergent in the general two-operator setting. In this work, we show that this uniqueness does not hold in the convex optimization case: when the operators are subdifferentials of proper, closed, convex functions, a strictly larger class of frugal, no-lifting resolvent-splitting methods is unconditionally convergent. We provide a complete characterization of all such methods in the convex optimization setting and prove that this characterization is sharp: unconditional convergence holds exactly on the identified parameter regions.
These results immediately yield new families of convergent ADMM-type and Chambolle--Pock-type methods obtained through their Douglas–Rachford reformulations.
\end{abstract}
\section{Introduction}
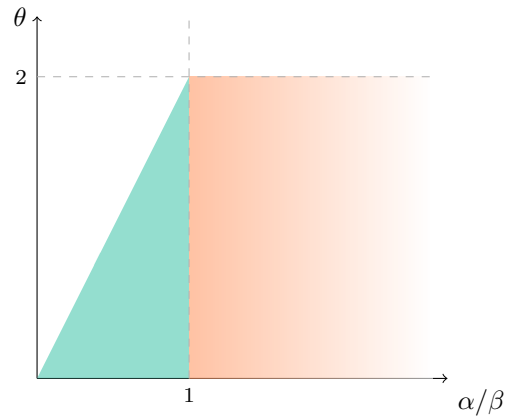
\begin{wrapfigure}{r}{0.5\linewidth}
    \vspace{-4em}
    \hspace{0em}
    \begin{tikzpicture}[scale=2, xshift=2cm]
        \draw[->] (0,0) -- (2.7,0) node[below right, font=\small] {$\alpha/\beta$};
        \draw[->] (0,0) -- (0,2.4) node[left, font=\small] {$\theta$};
    
        \definecolor{niceTeal}{RGB}{77,201,176}
        \definecolor{niceCoral}{RGB}{255,153,102} 
    
        \fill[niceTeal, opacity=0.6]
            (0,0) -- plot[domain=0:1] (\x, {2*\x}) -- (1,0) -- cycle;
        \shade[left color=niceCoral, right color=white, opacity=0.6]
            (1,0) -- (2.6,0) -- (2.6,2) -- (1,2) -- cycle;
    
        \draw[dashed, gray!60] (1,0) -- (1,2.4);
        \draw[dashed, gray!60] (0,2) -- (2.6,2);
        \node[left, font=\scriptsize] at (0,2) {$2$};
        \node[font=\scriptsize, below] at (1,0) {$1$};
    \end{tikzpicture}
    \caption{Parameter region of unconditional convergence for the algorithm in \eqref{eq::main_algorithm_ryu_functions}. The region of unconditional convergence for the general monotone inclusion setting (i.e., for the algorithm in \eqref{eq::main_algorithm_ryu}) is the line segment $\alpha/\beta=1$ and $\theta\in(0,2)$.}
    \label{fig::regions_S1_S2}
\end{wrapfigure}

The Douglas--Rachford splitting method is a celebrated algorithm for solving monotone inclusion problems of the form
\begin{equation}\label{eq::main_resolvent_problem}
\text{find } x \in \Hil \text{ such that } 0 \in A(x) + B(x),
\end{equation}
where $A,B:\Hil\rightrightarrows\Hil$ are maximally monotone operators and $\Hil$ is a real Hilbert space. While originally developed for solving heat conduction problems~\cite{DouglasRachford1956}, it has since Lions and Mercier in~\cite{lions1979splitting} proved its applicability to maximal monotone inclusion problems been extensively used and analyzed. A particularly important case of study is the convex optimization setting with problems of the form
\begin{equation}\label{eq::main_optimization_problem}
    \minimize_{x \in \Hil} f(x) + g(x),
\end{equation} 
where $f, g : \Hil \to \R\cup\{+\infty\}$ are proper, closed, and convex functions. Under mild assumptions (see e.g., \cite{BauschkeCombettes2017}), the problem \eqref{eq::main_optimization_problem} is a specialization of \eqref{eq::main_resolvent_problem} by letting $A=\partial f$ and $B=\partial g$.

An early interpretation of the Douglas--Rachford method was provided in~\cite{Eckstein1989,EcksteinBertsekas1992}, where it was shown to be an instance of the proximal point algorithm~\cite{Martinet1970,Rockafellar1976PPA}, offering a powerful framework for its analysis. Many subsequent works have provided (local) linear convergence for the Douglas--Rachford splitting method under various assumptions. For instance: \cite{BorweinSims2011AbsenceConvexity,HesseLuke2013NonconvexRegularity,BauschkeEtAl2014FriedrichsAngle,HesseLukeNeumann2014SparseAffine} for convex and nonconvex feasibility problems, 
\cite{Boley2013LocalLinearADMM,DemanetZhang2016Eventual,LiangFadiliPeyre2017Local} for convex optimization problems under more or less general partial smoothness assumptions, \cite{lions1979splitting,GiselssonBoyd2017MetricSelection,Giselsson2017TightDR,MoursiVandenberghe2019LipschitzStrongMono,RyuTaylorBergelingGiselsson2020OSPEP} under different combinations of strong monotonicity and smoothness assumptions in monotone inclusion and convex optimization settings, and \cite{LiPong2016NonconvexDR,ThemelisPatrinos2020TightNonconvex} for nonconvex optimization. Many of these analyses are tight in the sense that they are not improvable without changing the assumptions. 

The remarkable success of the Douglas–Rachford splitting method can perhaps in part be explained by the result of \cite{ryu2020uniqueness}, which shows that it is the unique ``simple'' resolvent splitting method for problems of the form \eqref{eq::main_resolvent_problem} that both possesses the \emph{fixed-point encoding} property and is \emph{unconditionally convergent}. Here, ``simple'' refers to methods that use \emph{minimal lifting} and are \emph{frugal}. Loosely speaking, a resolvent splitting method is an algorithm that accesses the operators $A$ and $B$ only through their resolvents (denoted $J_{\alpha A}$ and $J_{\beta B}$ respectively, where $\alpha,\beta\in \R_{++}$ are step sizes) and forms iterates from pre-defined linear combinations of their inputs and outputs. Such a method is \emph{frugal} if it evaluates each resolvent only once per iteration, has \emph{minimal lifting} if the dimension $d$ (the lifting number) of the product space vector in $\Hil^d$ that carries information between algorithm iterations is as small as possible (in the Douglas--Rachford case, we have no lifting, i.e., $d=1$), is a \emph{fixed-point encoding} if the fixed-point set of the algorithm and the solution set of the problem are simultaneously nonempty, and is \emph{unconditionally convergent} if there exist parameter choices guaranteeing convergence for every problem in the problem class and every algorithm starting point. 

The uniqueness of the Douglas–Rachford method under these conditions for problems of the form \eqref{eq::main_resolvent_problem} is established in \cite{ryu2020uniqueness} through two results: one parameterizing the class of frugal, no-lifting resolvent methods with the fixed-point encoding property, and one characterizing its subset of methods with unconditional convergence. It was shown that each frugal, no-lifting resolvent method for solving \eqref{eq::main_resolvent_problem}---up to equivalence---satisfies the fixed-point encoding property if and only if it is of the form
\begin{equation}\label{eq::main_algorithm_ryu}
\left\lfloor
\begin{aligned}
    x_1 &= J_{\alpha A}(z) \\
    x_2 &= J_{\beta B}\left(\left( 1 + \frac{\beta}{\alpha}\right)x_1  - \frac{\beta}{\alpha}z\right) \\
    z^+ &= z + \theta(x_2 - x_1)
\end{aligned} \right. ,
\end{equation} 
with parameters $\alpha,\beta\in\R_{++}$ and $\theta\in\R\setminus\{0\}$. The second result in \cite{ryu2020uniqueness} establishes that, among algorithms of the form \eqref{eq::main_algorithm_ryu}, the restrictions $\alpha=\beta$ and $\theta\in(0,2)$ are necessary and sufficient for unconditional convergence whenever $\dim(\Hil)\geq 2$. The results in \cite{ryu2020uniqueness} have sparked much research on minimal lifting methods, see, e.g., \cite{MalitskyTam2023MinimalLifting,AragonArtachoBotTorregrosaBelen2023MinimalLiftingPD,MorinBanertGiselsson2024Frugal,AkermanChencheneGiselssonNaldi2025SplitFB}, which address monotone inclusions involving sums of arbitrarily many operators.

In this work, we address the following question in the two-operator setting:
\begin{center}
\begin{minipage}{0.9\columnwidth}
\centering
\emph{Can we characterize all frugal, no-lifting, unconditionally convergent resolvent splitting methods with fixed-point encoding for solving \eqref{eq::main_resolvent_problem} in the convex optimization setting $A=\partial f$, $B=\partial g$?}
\end{minipage}
\end{center}
We show that the answer is \emph{yes}. Moreover, the restriction to the optimization setting leads to a significantly larger class of admissible algorithms than in the general monotone inclusion setting of \cite{ryu2020uniqueness}. In particular, we show that the convex-optimization specialization of \eqref{eq::main_algorithm_ryu}
\begin{equation}\label{eq::main_algorithm_ryu_functions}
\left\lfloor
\begin{aligned}
    x_1 &= \prox_{\alpha f}(z) \\
    x_2 &= \prox_{\beta g}\left(\left( 1 + \frac{\beta}{\alpha}\right)x_1  - \frac{\beta}{\alpha}z\right) \\
    z^+ &= z + \theta(x_2 - x_1)
\end{aligned} \right.,
\end{equation} 
completely characterizes all frugal, no-lifting, unconditionally convergent resolvent splittings with fixed-point encoding in the convex optimization setting \eqref{eq::main_optimization_problem}, when the parameters $(\alpha, \beta, \theta)$ satisfy 
\begin{align}
    (\alpha, \beta, \theta) \in \left\{(\alpha, \beta, \theta) \in  \R_{++}^3  \mid  \theta < \min\{2, 2\alpha/\beta\} \right\}.
    \label{eq:conv_param}
\end{align}
Figure~\ref{fig::regions_S1_S2} illustrates how the unconditional-convergence region expands in the convex optimization setting, shown in terms of the ratio $\alpha/\beta$, which in the monotone-inclusion setting is constrained to satisfy $\alpha/\beta=1$.

Unlike the Douglas--Rachford case, our sufficiency analysis does not rely on the conceptually simple proximal point algorithm or nonexpansive operators. Instead, we partition the region of unconditionally convergent parameters in \eqref{eq:conv_param} as $(\alpha, \beta, \theta) \in S^{(1)}\cup S^{(2)}$, where 
\begin{align*}
    S^{(1)} &\coloneqq \left\{(\alpha, \beta, \theta) \in \R_{++}^3  \mid \alpha/\beta \in (0, 1], \ \theta \in (0, 2\alpha/\beta) \right\}, \\
    S^{(2)} &\coloneqq \left\{(\alpha, \beta, \theta) \in \R_{++}^3  \mid \alpha/\beta \in [1, +\infty), \ \theta \in (0, 2) \right\},
    \end{align*}
and develop Lyapunov analyses for these newly identified parameter regions whose intersection $\left(S^{(1)} \cap S^{(2)} =\R_{++}\{(1,1)\}\times (0,2)\right)$ represents the unconditional-convergence region for monotone inclusions established in \cite{ryu2020uniqueness}. The discovery of the Lyapunov inequalities has been assisted by the automatic Lyapunov methodologies and tools of \cite{UpadhyayaBanertTaylorGiselsson2025AutoLyapTight,UpadhyayaTaylorBanertGiselsson2025AutoLyap}. Necessity is demonstrated using simple one-dimensional counterexamples involving the zero function and the indicator function of $\{0\}$.

While we primarily characterize the entire class of frugal, minimal lifting, unconditionally convergent resolvent splittings with fixed-point encoding in the optimization setting, we emphasize that all our convergence guarantees for methods with $\alpha/\beta\neq 1$ appear to be new. Because the alternating direction method of multipliers (ADMM) \cite{GlowinskiMarroco1975,GabayMercier1976,BoydParikhChuPeleatoEckstein2011} arises as a special case of Douglas–Rachford splitting \cite{Gabay1983,EcksteinBertsekas1992}, our results yield a family of completely new convergent ADMM-type methods for convex optimization. Likewise, since the Chambolle--Pock algorithm \cite{chambolle2011first} is also a Douglas–Rachford instance \cite{o2020equivalence}, we obtain an expanded class of convergent Chambolle--Pock–type methods, some of which are new. In particular, we present the, to the authors' knowledge, first doubly relaxed Chambolle--Pock-type method (see \cref{alg::extended_chambolle_pock} for the specific relaxations) with freedom in both relaxation parameters. Previous works, such as \cite{he2012convergence,banert2025chambolle}, have one of them fixed. We also present extensions of the parallel splitting method in \cite[Proposition~28.7]{BauschkeCombettes2017} and a generalized alternating projections method \cite{GAP_Agmon,GAP_Motzkin,GAP_Bregman}.

The rest of the paper is organized as follows. \Cref{sec::Preliminaries} introduces notation and definitions and presents a few preliminary results.
\Cref{sec::Problem_Setting} presents the problem setting and our main result; the complete characterization of all frugal, no-lifting resolvent splittings in the convex optimization setting that are both unconditionally convergent and have the fixed-point encoding property. \Cref{sec::Convergence Analysis} is devoted to proving the main result. \cref{sec::Additional_convergence_results} shows ergodic $\mathcal{O}(1/K)$-convergence for two primal-dual gap functions for the characterized methods.  \Cref{sec::admm_cp} applies our method to specific problem instances to derive new convergent ADMM-, Chambolle--Pock\mbox{-,} parallel splitting-, and alternating projection-type algorithms, and \cref{sec::conclusions} concludes the paper.
\section{Preliminaries} \label{sec::Preliminaries}
\subsection{Notation}
Our notation closely follows \cite{rockafellar1970convex,BauschkeCombettes2017}. We let $\N = \{0, 1, \dots\}$ denote the natural numbers, $\R$ denote the real numbers, $\R_+$ the nonnegative numbers, and $\R_{++}$ the positive numbers. Moreover $\R^n$ denotes the $n$-dimensional real vectors, $\R^{n\times m}$ denotes the real $n\times m$ matrices, and $\mathbb{S}^n$ denotes the real symmetric $n\times n$ matrices. 

We let $(\mathcal{H}, \langle \cdot, \cdot \rangle)$ be a real Hilbert space. The notation $A : \Hil \rightrightarrows \Hil$ means that $A$ is set-valued and maps $\Hil$ into subsets of $\Hil$. Its graph is defined as $\gra(A) \coloneqq \{(x,u) \mid u\in Ax\}$.
An operator $A:\Hil\rightrightarrows\Hil$ is monotone if $\langle u-v,x-y\rangle\geq 0$ for all $x,y\in\Hil$, $u\in A(x)$, and $v\in A(y)$. It is maximally monotone if there exists no monotone operator $B:\Hil\rightrightarrows\Hil$ such that $\gra(B)$ properly contains $\gra(A)$.  
The resolvent of $A:\Hil\rightrightarrows\Hil$ is given by $J_{A} \coloneqq (A + \Id)^{-1} : \Hil \rightrightarrows \Hil$ and is single-valued for all maximally monotone operators, see \cite[Corollary~23.9]{BauschkeCombettes2017}. Since we work exclusively with resolvents of maximally monotone operators $A$, we regard $J_A$ as a mapping $J_A : \Hil \to \Hil$. 
For a proper, closed, and convex function $f : \Hil \to \R \cup \{+\infty\}$, its subdifferential is the set-valued operator $\partial f : \Hil \rightrightarrows \Hil$ defined by
$$\partial f(x)
\coloneqq
\big\{ u \in \Hil \mid
f(y) \ge f(x) + \langle u, y - x \rangle
 \text{ for all } y \in \Hil
\big\}
$$ for each $x\in\Hil$.
The subdifferential $\partial f$ is maximally monotone \cite[Theorem~20.25]{BauschkeCombettes2017}, and its resolvent coincides with the proximal operator of $f$, i.e., $\prox_f = J_{\partial f}$; see \cite[Chapter~24]{BauschkeCombettes2017} for details.

Let $C$ be a subset of $\Hil$. We let $\iota_C : \Hil \to \R \cup \{+\infty\}$ denote the indicator function of $C$, which equals 0 for $x\in C$ and $+\infty$ for $x\not\in C$, and let $\Pi_C:=\prox_{\iota_C}$ denote the orthogonal projection onto $C$.

Given a matrix $Q = (q_{ij})_{i, j} \in \R^{n \times m}$, we associate it with the bounded linear operator $Q \otimes \Id : \Hil^m \to \Hil^n$. Throughout, we use the same notation for the matrix and the operator in the following sense: given a matrix $Q\in\R^{n\times m}$ and a vector $v = (v_1, \dots, v_m) \in \Hil^m$, then $Qv \coloneqq (Q \otimes \Id) v = \left(\sum_{j = 1}^m q_{1j}v_j, \dots, \sum_{j = 1}^m q_{nj}v_j\right) \in \Hil^n$. 

Moreover, we will overload $\langle\cdot, \cdot \rangle$ to also denote the induced inner product on $\Hil^m$, i.e., for all $(v, w)  \in \Hil^m \times \Hil^m$, we let $\langle v, w \rangle \coloneqq \sum_{i = 1}^m \langle v_i, w_i \rangle$. For any symmetric matrix $Q \in \Sym^{m}$ and any $v \in \Hil^m$, we denote $\mathcal{Q}(Q, v)  \coloneqq \langle v, Q v\rangle$.

We will use the following notations extensively throughout the paper. The set of maximally monotone operators is denoted by
\begin{equation}
    \mathcal{A} \coloneqq \{ A : \Hil \rightrightarrows \Hil \mid A \text{ is maximally monotone} \},
    \label{eq::max_mono_set}
\end{equation}
the set of proper, closed, and convex functions by
\begin{equation*}
    \mathcal{F} \coloneqq \{ f : \Hil \to \R \cup \{+\infty\} \mid f \text{ is proper, closed, and convex} \},
\end{equation*}
and the corresponding subdifferential set by
\begin{equation*}
    \partial \mathcal{F} \coloneqq \{ \partial f :\Hil\rightrightarrows\Hil  \mid  f \in \mathcal{F}\}.
\end{equation*}
We note that $\partial \mathcal{F} \subset \mathcal{A}$ by \cite[Theorem~20.25]{BauschkeCombettes2017}. Throughout this paper, the Hilbert space $\Hil$ is assumed to satisfy $\dim \Hil \geq 1$, i.e., $\Hil$ should not be a singleton.

\subsection{Definitions and Preliminary Results}

\begin{definition}[$D_f$] \label{def::D_h}
    Let $f:\Hil\to\R\cup\{+\infty\}$. Then $D_f : (\dom f)^2 \times \Hil \to \R$ is defined by $$D_f(x, y, u) \coloneqq f(x) - f(y) - \langle u, x - y \rangle$$ for each $(x, y) \in (\dom f)^2$ and $u \in \Hil$.
\end{definition}
A function $D_f$ satisfying \cref{def::D_h} will be called a \emph{Bregman-type distance}. 
\begin{lemma}
    \label{lem::Bregman_three_point}
    Let $f:\Hil\rightarrow \R\cup \{+\infty\}$.
    Then
    $$D_f(z, y, v) - D_f(x, y, v) = D_f(z, x, u) + \langle u-v, z-x\rangle,
    $$ for each $(x, y, z)\in (\dom f)^3$ and $(u, v)\in \Hil^2$.
\end{lemma}
\begin{proof}
    By \cref{def::D_h} we have that
    \begin{align*}
        D_f(z, y, v) - D_f(x, y, v) &= \left(f(z) - f(y) - \langle v, z-y\rangle\right) - \left(f(x) - f(y) - \langle v, x-y\rangle\right)\\
        &=f(z)-f(x) - \langle v, z-x\rangle \\
        &= D_f(z, x, u) + \langle u-v, z-x\rangle.
    \end{align*}
\end{proof}
\begin{lemma}
    \label{lem::Bregman_two_point}
    Let $f:\Hil\rightarrow \R\cup \{+\infty\}$.
    Then
    $$D_f(x, y, v) + D_f(y, x, u) = \langle u-v, x-y\rangle
    $$ for each $(x,y)\in (\dom f)^2$ and $(u, v)\in \Hil^2$.
\end{lemma}
\begin{proof}
    Let $z = y$ in \cref{lem::Bregman_three_point} and note that $D_{f}(y,y,v)=0$.
\end{proof}

\begin{lemma} \label{lem::Dh_nonnegative}
    Let $f:\Hil\to\R\cup\{+\infty\}$. Then $D_f(x, y, u) \geq 0$ for each $(x, y) \in (\dom f)^2$ and $u \in \partial f(y)$.
\end{lemma}
\begin{proof}
    This follows directly from the subdifferential definition \cite[Definition~16.1]{BauschkeCombettes2017}. 
\end{proof}

\section{Problem Setting and Main Results} \label{sec::Problem_Setting}

In this section, we present our main result: a complete characterization of the frugal, no-lifting resolvent-splitting methods that are unconditionally convergent and satisfy the fixed-point encoding property for problems of the form \eqref{eq::main_resolvent_problem} with $A=\partial f$ and $B=\partial g$, where $f,g:\Hil\to\R\cup\{+\infty\}$ are proper, closed, and convex functions. The starting point for our analysis is Theorem~1 of \cite{ryu2020uniqueness}, which gives an explicit parameterization of all frugal, no-lifting resolvent-splitting methods satisfying the fixed-point encoding property, under the assumption that the two operators are maximally monotone. Since the proof of that result relies only on subdifferential operators of proper, closed, convex functions, the same parameterization applies in our setting. This parameterization is given explicitly by \cref{alg::main}, and hence the task of obtaining our characterization reduces to determining which choices of $(\alpha,\beta,\theta)$ in \cref{alg::main} lead to unconditional convergence in the convex optimization setting.

To avoid unnecessary notational overhead, we refer the reader to \cite{ryu2020uniqueness} for the formal definition of the class of frugal, no-lifting resolvent-splitting methods and for the proof of its equivalence to the parameterization in \cref{alg::main}. For convenience, we reproduce the characterization result here in a form relevant to our convex optimization setting.

\begin{algorithm}[htbp]
	\caption{Extended Douglas--Rachford Splitting}
	\begin{algorithmic}[1]
	    \State \textbf{Input:} $(A, B, z^0, \alpha, \beta, \theta) \in  \mathcal{A} \times \mathcal{A} \times \Hil \times \R_{++} \times \R_{++} \times (\R\setminus\{0\})$
		\For {$k=0,1,2,\ldots$}
		    \State \textbf{Update: } \begin{equation*}
                \left\lfloor
                \begin{aligned}
                    x^k_1 &= J_{\alpha A}(z^k) \\
                    x^k_2 &= J_{\beta B}\left(\left( 1 + \frac{\beta}{\alpha}\right)x^k_1  - \frac{\beta}{\alpha}z^k\right) \\
                    z^{k+1} &= z^k + \theta(x^k_2 - x^k_1)
                \end{aligned} \right. 
                \end{equation*}
		\EndFor
	\end{algorithmic}
\label{alg::main}
\end{algorithm}
\begin{theorem}[From \cite{ryu2020uniqueness} Theorem~1]
    \label{thm::Full_characterization}
            Let $f,g:\Hil\to\R\cup\{+\infty\}$ be proper, closed, and convex functions. Up to equivalence, a method is a frugal, no-lifting resolvent-splitting satisfying the fixed-point encoding property for solving \cref{eq::main_resolvent_problem} with $A=\partial f$ and $B=\partial g$ if and only if it is of the form given in \cref{alg::main}.
\end{theorem}
The fixed-point encoding property ensures that solutions of the inclusion problem \eqref{eq::main_resolvent_problem} correspond exactly to fixed points of \cref{alg::main}. Since unconditional convergence is defined in terms of convergence to such fixed points, we begin by making the fixed-point set of \cref{alg::main} explicit.

\begin{definition}[$\Fix(A, B, \alpha, \beta)$] \label{def::fix}
    For any $(A, B, \alpha, \beta) \in \mathcal{A}^2 \times \R_{++}^2$, let 
    $$\Fix(A, B, \alpha, \beta) \coloneqq \left\{(x^\star, x^\star, z^\star) \in \Hil^3 \mid x^\star = J_{\alpha A}(z^\star) = J_{\beta B}\left(\left(1+\frac{\beta}{\alpha}\right)x^\star - \frac{\beta}{\alpha}z^\star\right) \right\}.$$ 
\end{definition}
For ease of reference, we also explicitly state the equivalence between the solution set of the inclusion \eqref{eq::main_resolvent_problem} and the fixed-point set of \cref{alg::main}.
In particular, this shows that whenever $(x^\star,x^\star,z^\star)\in\Fix(A,B,\alpha,\beta)$, the point $x^\star$ must be a solution of the corresponding inclusion \eqref{eq::main_resolvent_problem}.
\begin{proposition} \label{prp::fix_point_prp}
    Let $(A, B, \alpha, \beta) \in \mathcal{A}^2 \times \R_{++}^2$ and $x^\star \in \Hil$. Then, $x^\star \in \zer(A+B)$ if and only if there exists $z^\star \in \Hil$ such that $(x^\star, x^\star, z^\star) \in \Fix(A, B, \alpha, \beta)$.
\end{proposition}
\begin{proof}
    We have that $x^\star \in \zer(A+B)$ if and only if there exists some $z^\star \in \Hil$ such that
    \begin{equation*}
    \begin{aligned}
        &\begin{cases}
            \frac{1}{\alpha}(z^\star - x^\star) &\in A(x^\star)\\
            \frac{1}{\alpha}(x^\star - z^\star) &\in B(x^\star)
        \end{cases} \\
        \iff &\begin{cases}
            \frac{1}{\alpha}(z^\star - x^\star) &\in A(x^\star)\\
            \frac{1}{\beta}\left( \left( 1 + \frac{\beta}{\alpha}\right)x^\star - \frac{\beta}{\alpha}z^\star - x^\star \right) &\in B(x^\star)
        \end{cases} \\
        \iff &\begin{cases}
            x^\star &= J_{\alpha A}(z^\star) \\
            x^\star &= J_{\beta B}\left(\left(1+\frac{\beta}{\alpha}\right)x^\star - \frac{\beta}{\alpha}z^\star\right)
        \end{cases} \\
        \iff &(x^\star, x^\star, z^\star) \in \Fix(A, B, \alpha, \beta).
    \end{aligned}
    \end{equation*}
\end{proof}

We now formalize what it means for \cref{alg::main} to converge unconditionally for a given parameter choice over a subset of operator pairs from $\mathcal{A}$, the class of maximally monotone operators introduced in \eqref{eq::max_mono_set}.

\begin{definition}\label{def::unconditional}
Let $\mathcal{C} \subset \mathcal{A} \times \mathcal{A}$ and $(\alpha,\beta,\theta) \in \R_{++}^2 \times (\R\setminus\{0\})$.
We say that \cref{alg::main} \textbf{converges unconditionally for $(\alpha,\beta,\theta)$ over $\mathcal{C}$} if, for every $z^0\in\Hil$ and every $(A,B)\in\mathcal{C}$ with $\zer(A+B)\neq\emptyset$, the sequence $\bigl((x_1^k,x_2^k,z^k)\bigr)_{k\in\N}$ generated by \cref{alg::main} applied to $(A, B, z^0, \alpha, \beta, \theta)$ converges weakly to a point $(x^\star,x^\star,z^\star)\in\Fix(A, B, \alpha, \beta)$.
\end{definition}

We now introduce a convenient notation for describing the set of parameters in \cref{alg::main} that guarantee unconditional convergence for a given class of operator pairs. This will allow us to state our results in a compact form.

\begin{definition}[$S(\mathcal{C})$]
    Let $\mathcal{C} \subset \mathcal{A} \times \mathcal{A}$. Define $$S(\mathcal{C}) \coloneqq \{(\alpha, \beta, \theta) \in \R^2_{++} \times (\R\setminus\{0\}) \mid \cref{alg::main} \text{ converges unconditionally for } (\alpha, \beta, \theta) \text{ over } \mathcal{C}\}.$$
\end{definition}

With this notation in place, we can state the following key result from \cite[Theorem~2]{ryu2020uniqueness}.

\begin{theorem}[\cite{ryu2020uniqueness} Theorem 2]\label{thm::ryu}
   If $\dim \mathcal{H} \geq 2$, then \cref{alg::main} converges unconditionally for $(\alpha, \beta, \theta) \in \R^2_{++} \times (\R \setminus \{0\})$ over $\mathcal{A} \times \mathcal{A}$ if and only if $\alpha = \beta$ and $\theta \in (0, 2)$, i.e., $$S(\mathcal{A} \times \mathcal{A}) = \R_{++}\{(1, 1)\} \times (0, 2).$$
\end{theorem}

Before stating our main result in \cref{thm::main_theorem}, we introduce the following two parameter sets, whose union will be shown to constitute the full region of unconditional convergence in the convex optimization setting. This division is used because our proof treats the two cases separately. The extent to which this enlarges the admissible parameter region compared with the maximal monotone setting of \cite{ryu2020uniqueness} is illustrated in \cref{fig::regions_S1_S2}.

\begin{definition}[$S^{(i)}$]
    Let $S^{(1)}, S^{(2)} \subset \R^3$ be defined by
    \begin{align*}
        S^{(1)} &\coloneqq \left\{(\alpha, \beta, \theta) \in \R_{++}^3 \mid \alpha/\beta \in (0, 1], \ \theta \in (0, 2\alpha/\beta) \right\}, \\
        S^{(2)} &\coloneqq \left\{(\alpha, \beta, \theta) \in \R_{++}^3 \mid \alpha/\beta \in [1, +\infty), \ \theta \in (0, 2) \right\}.
    \end{align*}
\end{definition}

We are now ready to state our main result.

\begin{theorem} \label{thm::main_theorem}
    \cref{alg::main} converges unconditionally for $(\alpha, \beta, \theta) \in \R^2_{++} \times (\R \setminus \{0\})$ over $\partial\mathcal{F} \times \partial\mathcal{F}$ if and only if $\theta  \in (0, \min \{2, 2\alpha/\beta\})$. 
    Moreover,
    $$S(\partial\mathcal{F} \times \partial\mathcal{F}) = S(\partial \mathcal{F}\times \mathcal{A}) = S^{(1)}  \cup S^{(2)}.$$
\end{theorem} 
Note that this result applies not only to the convex optimization setting but also to the mixed case in which 
$A$ is a subdifferential operator of a proper, closed, and convex function while $B$ is an arbitrary maximally monotone operator.

In contrast to the maximal monotone setting, no additional assumptions on the dimension of $\Hil$ are needed when restricting to $\partial\mathcal{F} \times \partial\mathcal{F}$ or $\partial\mathcal{F} \times \mathcal{A}$. In particular, \cref{thm::main_theorem} also covers the one-dimensional case. In light of the fact that every maximally monotone operator on $\R$
is the subdifferential of a proper, closed, and convex function (see \cite[Corollary~22.23]{BauschkeCombettes2017}) and since any one-dimensional $\Hil$ is isometrically isomorphic to $\R$, with a mapping that preserves maximal monotonicity and subdifferentials, \cref{thm::main_theorem} allows us to revisit \cite[Theorem~2]{ryu2020uniqueness} and close the gap left open in the case $\dim\Hil=1$. 

\begin{corollary}\label{cor::main_corollary}
    If $\Hil = \R$, then $S(\mathcal{A} \times \mathcal{A}) = S(\partial \mathcal{F} \times \partial \mathcal{F}) = S^{(1)} \cup S^{(2)}$.
\end{corollary}
\begin{proof}
    Combine \cref{thm::main_theorem} with the fact that when $\Hil = \R$, then $\mathcal{A} = \partial \mathcal{F}$, see \cite[Corollary~22.23]{BauschkeCombettes2017}.
\end{proof}

Before we proceed to prove \cref{thm::main_theorem}, we note that the iteration in \cref{alg::main} has been analyzed before. In particular, in \cite{DaoPhan2019AdaptiveDRS} it was written as
\begin{align}
    z^{k+1} = T_{A,B}^{\alpha,\beta,\theta}(z^k) \qquad\quad\text{where}\quad\qquad  T_{A,B}^{\alpha,\beta,\theta} \coloneqq \left(1-\frac{\theta\beta}{\alpha+\beta}\right) \Id + \frac{\theta\beta}{\alpha+\beta}R_{B}^{\beta,\alpha}R_A^{\alpha,\beta},
    \label{eq::alg_composition_rem}
\end{align}
and $R_{A}^{\alpha,\beta}:\Hil\to\Hil$ and $R_{B}^{\beta,\alpha}:\Hil\to\Hil$ are \emph{relaxed resolvents} defined as
\begin{align}
    R_{A}^{\alpha,\beta}  \coloneqq \left(1+\frac{\beta}{\alpha}\right)J_{\alpha A}-\frac{\beta}{\alpha}\Id \qquad\quad{\text{and}}\qquad\quad R_{B}^{\beta,\alpha}  \coloneqq \left(1+\frac{\alpha}{\beta}\right)J_{\beta B}-\frac{\alpha}{\beta}\Id.
    \label{eq::reflected_res_rem}
\end{align}
This formulation has been analyzed, e.g., in \cite{DaoPhan2019AdaptiveDRS,GiselssonMoursi2021Compositions,BartzDaoPhan2022ConicalAveragedness,Dao_2025} in different settings than ours.
Under the parameter restrictions in \cref{thm::main_theorem}, we have $\frac{\theta\beta}{\alpha+\beta}\in(0,1)$, so \eqref{eq::alg_composition_rem} is a $\frac{\theta\beta}{\alpha+\beta}$-averaged iteration of the composition $R_{B}^{\beta,\alpha}R_A^{\alpha,\beta}$. 
It turns out that $T_{A,B}^{\alpha,\beta,\theta}$ is an averaged operator \cite[Definition~4.33]{BauschkeCombettes2017} for all $(A,B)\in\mathcal{A}\times\mathcal{A}$ if and only if $\alpha=\beta$---in which case the two relaxed resolvents are nonexpansive---and it can fail to be averaged even for the subclass $\partial\mathcal{F}\times\partial\mathcal{F}$ unless $\alpha=\beta$. Consequently, in contrast to the standard Douglas--Rachford case, the proof of \cref{thm::main_theorem} cannot rely on averaged-operator theory and instead makes use of a Lyapunov analysis, as developed in \cref{sec::Convergence Analysis}.
\section{Convergence Analysis} \label{sec::Convergence Analysis}
This section is devoted to the proof of \cref{thm::main_theorem}.
In \cref{sec::Lyapunov_analysis}, we establish sufficiency of the parameter conditions by means of a Lyapunov analysis, showing that
\begin{align*}
S^{(1)} \cup S^{(2)} \subset S(\partial\mathcal{F}\times \mathcal{A}) \subset S(\partial\mathcal{F}\times \partial\mathcal{F}).
\end{align*}
In \cref{sec::Counter_Examples}, we prove necessity by constructing explicit counterexamples for which the algorithm does not converge outside these regions, thereby showing that
\begin{align*}
S(\partial\mathcal{F}\times \partial\mathcal{F}) \subset S^{(1)} \cup S^{(2)}.
\end{align*}
Finally, in \cref{sec::thm_proof}, we combine these results to conclude the proof of \cref{thm::main_theorem}.

\subsection{Lyapunov Analysis} \label{sec::Lyapunov_analysis}
The Lyapunov analysis will be carried out over the class $\partial \mathcal{F} \times \mathcal{A}$; that is, we consider inclusion problems of the form
$$\text{find } x\in \mathcal{H} \text{ such that } 0 \in \partial f(x) + B(x),$$
where $f\in \mathcal{F}$ and $B\in \mathcal{A}$. This includes inclusion problems over $\partial \mathcal{F} \times\partial \mathcal{F}$ as a special case. Throughout \cref{sec::Lyapunov_analysis} we will assume the following.
\begin{assumption}
    \label{ass::Lyapunov_assumption}
    Suppose that
    \begin{enumerate}
    [label=\textit{(\roman*})]
        \item      \label{ass::Lyapunov_assumption_i}
$(f, B) \in \mathcal{F} \times \mathcal{A}$ such that $\zer(\partial f + B) \neq \emptyset$,
        \item $(\alpha, \beta, \theta)\in \R_{++}^2 \times (\R \setminus \{0\})$.
    \end{enumerate}
\end{assumption}
Let us also define the sequences and points needed for the upcoming results.
\begin{definition}
    \label{def::algorithm_sequences}
    Under \cref{ass::Lyapunov_assumption},
    \begin{enumerate}
    [label=\textit{(\roman*})]
        \item let $(x^\star, x^\star, z^\star) \in \Fix(\partial f, B, \alpha, \beta)$, where $\Fix(\partial f, B, \alpha, \beta)$ is nonempty by
        \cref{prp::fix_point_prp} and \itemcref{ass::Lyapunov_assumption}{ass::Lyapunov_assumption_i}, and let $u^\star \coloneqq \frac{1}{\alpha}(z^\star - x^\star)$,
        \item let the sequence $\bigl((x^k_1, x^k_2, z^k)\bigr)_{k \in \N}$ be generated by \cref{alg::main} applied to $(\partial f, B, z^0, \alpha, \beta, \theta)$ for some initial point $z^0\in \Hil$,
        \item and let $(u_1^k)_{k\in \N}$
        and $(u_2^k)_{k\in \N}$ be sequences such that
        \begin{align*}
            u_1^k &= \frac{1}{\alpha}(z^k - x^k_1),\\
            u^k_2 &= \frac{1}{\beta}\left( \left( 1 + \frac{\beta}{\alpha}\right) x^k_1 - \frac{\beta}{\alpha}z^k - x^k_2\right),
        \end{align*}
        for all $k\in \N$.\\
    \end{enumerate}
\end{definition}
\begin{remark}
\label{rem::subdifferential}
   For the sequences of \cref{def::algorithm_sequences}, we have that  $u_1^k\in \partial f(x_1^k)$ and $u_2^k\in B(x_2^k)$ for all $k\in \N$, due to the algorithm update of \cref{alg::main}. We additionally have that $u^\star \in \partial f(x^\star)$ and $-u^\star \in B(x^\star)$.
\end{remark}

Next, we introduce the \textit{Lyapunov function} $V_k$ and the \textit{residual function} $R_k$, both of which contain Bregman-type distance terms $D_f$. 
Since $\ran J_{\alpha \partial f} = \dom \partial f \subset \dom f$, by \cite[Proposition~23.2(i), Proposition~16.4(i)]{BauschkeCombettes2017}, these Bregman-type distances are well-defined and nonnegative for all $k \in \N$ according to \cref{def::D_h,lem::Dh_nonnegative,def::algorithm_sequences}.

\begin{definition}[$Q^{(1)}_k, V_k, R_k$]\label{def::V_R}
   For the sequences of \cref{def::algorithm_sequences} and for all $k\in \N$, let $\tau \coloneqq \alpha/\beta$ and define
    \begin{align*}
        Q^{(1)}_k &\coloneqq  \norm{z^k-z^\star + \frac{\theta}{2}(x_2^k-x_1^k) + (\tau-1)(x_1^k-x^\star)}^2 + \frac{\theta(4\tau-\theta)}{4} \norm{x_2^k-x_1^k }^2 \\
        &\quad\; + \tau(1-\tau) \norm{x_1^k-x^\star}^2, \\
        V_k &\coloneqq Q^{(1)}_k + 2\alpha(1 - \tau)D_f(x^k_1, x^\star, u^\star), \\
        R_k &\coloneqq \theta(2\tau-\theta)\norm{x_1^k-x_2^k }^2 + (1-\tau)\norm{x_1^{k+1}-x_1^k}^2 + 2\alpha(1 - \tau)D_f(x^k_1, x^{k+1}_1, u^{k+1}_1).
        \end{align*}
\end{definition}     
Nonnegativity of $V_k$ and $R_k$ for all $(\alpha, \beta, \theta)\in S^{(1)}$ is immediate. The following lemma will be used in \cref{lemma::analysis} to show that $V_k$ and $R_k$ are nonnegative for $(\alpha, \beta, \theta)\in S^{(2)}$.

\begin{lemma} \label{lem::Lyapunov_def_alternative_2}
    For the sequences of \cref{def::algorithm_sequences} and for all $k\in \N$, let $\tau \coloneqq \alpha/\beta$ and let
        \begin{align*}
            Q^{(2)}_k &\coloneqq \norm{z^k-z^\star +\frac{\theta}{2}(x_2^k-x_1^k)}^2 +\frac{\theta(4\tau-\theta)}{4}\norm{x_2^k-x_1^k +  \frac{2(\tau-1)}{4\tau-\theta}(x_1^k-x^\star)}^2 \\
        &\quad\; + \frac{\tau(\tau-1)(4-\theta)}{4\tau-\theta}\norm{x_1^k-x^\star}^2.
        \end{align*}
    If $(\alpha, \beta, \theta)\in S^{(1)}\cup S^{(2)}$, then $V_k$ and $R_k$ satisfy
    \begin{absolutelynopagebreak}
    \begin{align}
          V_k &= Q^{(2)}_k + 2\alpha(\tau - 1) D_f(x^\star, x^k_1, u^k_1),\tag{\textbf{VR}} \label{eq::lem_equations}\\
          R_k &= (\tau-1)\norm{x_1^{k+1}-x_1^k + \theta(x_1^k-x_2^k)}^2 + \tau\theta(2-\theta)\norm{x_1^k-x_2^k}^2 + 2\alpha(\tau - 1)D_f(x^{k+1}_1, x^{k}_1, u^{k}_1)\notag
      \end{align}
      \end{absolutelynopagebreak}
        for all $k \in \N$.
\end{lemma}
\begin{proof}   
    Note first that if $(\alpha, \beta, \theta) \in S^{(1)}\cup S^{(2)}$, then $4\alpha/\beta>\theta$. This implies, in particular, that the definition of $Q_k^{(2)}$ never involves division by $0$. The equalities \eqref{eq::lem_equations} are shown in \cref{subsec::Proof_of_Lemma_Lyapunov} and symbolically in\footnote{\label{fn:shared}\url{https://github.com/MaxNilsson-phd/Symbolic_verification_of_Extending_Douglas--Rachford_Splitting_for_Convex_Optimization}}.
\end{proof}

\begin{remark} \label{rem::quadratic_form_remark}
    Note that for each $i\in\{1,2\}$ there exists a unique matrix $\bar{Q}_V^{(i)} \in \Sym^{3}$ such that the equality $$Q^{(i)}_k = \mathcal{Q}(\bar{Q}_V^{(i)}, (x^k_1 - x^\star, x^k_2 - x^\star, z^k - z^\star))$$ holds for all $k \in \N$. See \cref{appendix} for the explicit expression for $\bar{Q}_V^{(i)}$, lifted to $\R^{5 \times 5}$. Viewing $\bar{Q}^{(i)}_k$ as a quadratic form in this way will simplify the proof of \cref{prp::weak_convergence}.
\end{remark}

Besides the Lyapunov and residual functions, we need the following quantity in order to state the Lyapunov equality.
\begin{definition}[$I_k$]
\label{def::I_k}
For the sequences of \cref{def::algorithm_sequences} and for all $k\in \N$, define
    $$
    I_k\coloneqq \langle u_1^k-u^\star, x_1^k- x^\star\rangle + \langle u_2^k+u^\star, x_2^k- x^\star\rangle+ \langle u_1^{k+1}-u^\star, x_1^{k+1}- x^\star\rangle+ \langle u_2^{k+1}+u^\star, x_2^{k+1}- x^\star\rangle.
    $$ 
\end{definition}
\begin{lemma}
\label{lem::I_k}
    The quantity $I_k$ in \cref{def::I_k} is nonnegative for all $k \in \N$.
\end{lemma}
\begin{proof}
    Let $k \in \N$. By \cref{rem::subdifferential} we get that each of the four inner products in $I_k$ is a monotonicity term; the first and third arise from the monotonicity of $\partial f$ evaluated at $(x_1^k,x^\star)$ and $(x_1^{k+1},x^\star)$, while the second and fourth arise from the monotonicity of $B$ evaluated at $(x_2^k,x^\star)$ and $(x_2^{k+1},x^\star)$. Therefore, $I_k\geq 0$.
\end{proof}

The following Lyapunov equality is the key component for establishing convergence of \cref{alg::main}.

\begin{proposition} \label{lemma::analysis}
     With $V_k$ and $R_k$ as in \cref{def::V_R} and $I_k$ as in \cref{def::I_k}, if $(\alpha, \beta, \theta) \in S^{(1)}\cup S^{(2)}$, then $V_k \geq 0$, $R_k \geq 0$, and the equality \begin{equation} \label{eq::lemma_equality}
     V_{k+1} = V_k - R_k -\theta\alpha I_k, \tag{\textbf{LE}}
     \end{equation} holds, for all $k \in \N$.
\end{proposition}
\begin{proof}
    The fact that $V_k$ and $R_k$, as defined in \cref{def::V_R}, both are nonnegative if $(\alpha, \beta, \theta)\in S^{(1)}$ follows immediately from their respective definitions, since all coefficients in front of the norms and the Bregman-type distances are nonnegative when $(\alpha, \beta, \theta) \in S^{(1)}$ and the Bregman-type distances are nonnegative by \cref{lem::Dh_nonnegative}.
    For $(\alpha, \beta, \theta)\in S^{(2)}$, nonnegativity follows from the same argument applied to the alternative expressions for $V_k$ and $R_k$ of \cref{lem::Lyapunov_def_alternative_2}. The Lyapunov equality \eqref{eq::lemma_equality} is verified in \cref{subsec::Proof_of_LE} and symbolically in\hyperref[fn:shared]{\textsuperscript{\ref{fn:shared}}}.
\end{proof}
The Lyapunov equality in \cref{lemma::analysis} allows us to draw the following conclusions.
\begin{proposition} \label{prp::analysis}
    With the sequences of \cref{def::algorithm_sequences}, $V_k$ and $R_k$ as in \cref{def::V_R} and $I_k$ as in \cref{def::I_k}, if $(\alpha, \beta, \theta) \in S^{(1)}\cup S^{(2)}$, then
    \begin{enumerate}[label=\textit{(\roman*})]
        \item $\left(R_k\right)_{k \in \N}$ and $\left(I_k\right)_{k\in \N}$ are summable,\label{prp::analysis_i}
        \item $x^k_1 - x^k_2 \to 0$ and $u^k_1 + u^k_2 \to 0$ as $k \to \infty$,\label{prp::analysis_ii}
        \item $(x^k_1)_{k \in \N}, (x^k_2)_{k \in \N}$, $(u_1^k)_{k\in \N}$, $(u_2^k)_{k\in \N}$ and $(z^k)_{k \in \N}$ are bounded,\label{prp::analysis_iii}
        \item $\left(Q_k^{(1)}\right)_{k \in \N}$ and $\left(Q_k^{(2)}\right)_{k\in\N}$ converge.\label{prp::analysis_iv}
    \end{enumerate}
\end{proposition}
\begin{proof}
    From \cref{lemma::analysis,lem::I_k} we conclude after telescoping \eqref{eq::lemma_equality} that, for all $K\in \N$:
    $$0 \leq \sum_{k = 0}^K\left(R_k + \theta \alpha I_k\right) \leq V_0 - V_{K+1} \leq V_0,$$  where we also use that $\alpha\theta>0$ for all $(\alpha,\beta,\theta)\in S^{(1)}\cup S^{(2)}$. By letting $K \to \infty$, we conclude that \ref{prp::analysis_i} holds. 

    Let $i\in \{1,2\}$ such that $(\alpha, \beta, \theta)\in S^{(i)}$. By \cref{lemma::analysis} and \ref{prp::analysis_i}, we have that $R_k \geq 0$ and $R_k \to 0$ as $k \to \infty$, respectively.
    If $i = 1$, by using the expression for $R_k$ from \cref{def::V_R}, nonnegativity of $(1-\tau)$ and $2\alpha (1-\tau)$, nonnegativity of the Bregman-type term by \cref{lem::Dh_nonnegative}, and strict positivity of $\theta(2\tau-\theta)$, we get that $x^k_1 - x^k_2 \to 0$ as $k \to \infty$.
    
    If $i = 2$, an analogous argument holds. Indeed, by using the expression for $R_k$ from \cref{lem::Lyapunov_def_alternative_2}, by nonnegativity of $(\tau - 1)$ and $2\alpha(\tau - 1)$, nonnegativity of the Bregman-type term by \cref{lem::Dh_nonnegative}, and strict positivity of $\tau\theta(2 - \theta)$, we get that $x^k_1 - x^k_2 \to 0$ as $k \to \infty$.
    
    By \cref{def::algorithm_sequences} we have that
    $$u_1^k + u_2^k = \frac{1}{\beta}(x_1^k-x_2^k) \rightarrow 0$$
    as $k \rightarrow \infty$ and so \ref{prp::analysis_ii} holds.
    
    Since $$0 \leq V_{k+1} \leq V_k- R_k \leq V_k$$ for all $k \in \N$, we conclude that $\left(V_k\right)_{k \in \N}$ is a sequence of nonincreasing nonnegative numbers and so $\left(V_k\right)_{k \in \N}$ converges.

    If $(\alpha, \beta, \theta)\in S^{(1)}$, then since all terms of $V_k$, given by \cref{def::V_R}, are nonnegative, we have in particular that $\tau(1-\tau)\norm{x_1^k-x^\star}^2$ is bounded, implying that $\left((1-\tau)(x_1^k-x^\star)\right)_{k\in\N}$ is bounded since $\tau > 0$. Since the sequence 
    \begin{align*}
        \left(\norm{z^k - z^\star + \frac{\theta}{2}(x_2^k-x_1^k) + (\tau-1)(x_1^k-x^\star) }^2\right)_{k\in \N}
    \end{align*}
    is bounded, this fact, combined with \ref{prp::analysis_ii}, shows that $(z^k-z^\star)_{k\in \N}$ is bounded.

If instead $(\alpha, \beta, \theta)\in S^{(2)}$, by a similar argument and using \cref{lem::Lyapunov_def_alternative_2}, we have that
\begin{align*}
    \left(\norm{z^k-z^\star + \frac{\theta}{2}(x_2^k-x_1^k)}^2\right)_{k\in \N}
\end{align*} is bounded, this fact, combined with \ref{prp::analysis_ii}, shows that $(z^k-z^\star)_{k\in \N}$ is bounded.

    By nonexpansiveness of resolvents, see \cite[Proposition~2.38]{BauschkeCombettes2017}, the definitions of $(x_1^k)_{k\in \N}$ and of the fixed-point set in \cref{def::fix}, and boundedness of $(z^k-z^\star)_{k\in \N}$, we get that $(x_1^k-x^\star)_{k\in \N}$ is bounded. Boundedness of $(x_2^k-x^\star)_{k\in \N}$ then follows from \itemcref{prp::analysis}{prp::analysis_ii}. Therefore, $\bigl((x^k_1, x^k_2, z^k)\bigr)_{k \in \N}$ is bounded and by \cref{def::algorithm_sequences} we conclude that $(u_1^k)_{k\in \N}$ and $(u_2^k)_{k\in \N}$ are bounded, so \ref{prp::analysis_iii} holds.

    Using $u^\star \in \Hil$ as in \cref{def::algorithm_sequences}, we get that
    \begin{equation*}
        \begin{aligned}
            0 &\leq D_f(x^\star, x_1^k, u_1^k) + D_f(x_1^k, x^\star, u^\star) \\
            &= \langle u_1^k -u^\star, x_1^k-x^\star \rangle \\
            &\leq \langle u_1^k -u^\star, x_1^k-x^\star \rangle + \langle u_2^k+ u^\star, x_2^k-x^\star\rangle\\
            &= -\langle u_1^k + u_2^k, x^\star\rangle + \langle u^\star, x_2^k-x_1^k\rangle + \langle u_1^k + u_2^k, x_1^k\rangle + \langle u_2^k, x^k_2-x_1^k\rangle \to 0
        \end{aligned}
    \end{equation*}
    as $k \to \infty$, where the first equality is due to \cref{lem::Bregman_two_point}, the second inequality is due to the monotonicity of $B$ since $u_2^k\in B(x_2^k)$ and $-u^\star \in B(x^\star)$, and the limit is a consequence of boundedness of $(x_1^k)_{k\in \mathbb N}$ and $(u_2^k)_{k\in \mathbb N}$ combined with \ref{prp::analysis_ii}.
    We conclude that $D_f(x^\star, x_1^k, u_1^k)\to 0$ and $D_f(x_1^k, x^\star, u^\star)\to 0$ as $k \to \infty$. Combining this with the fact that $\left(V_k\right)_{k \in \N}$ converges, we get in light of the expressions for $V_k$, as given in \cref{def::V_R} and \cref{lem::Lyapunov_def_alternative_2}, that \ref{prp::analysis_iv} holds.
\end{proof}

\begin{proposition}\label{prp::weak_convergence}
    \cref{alg::main} converges unconditionally for $(\alpha, \beta, \theta) \in S^{(1)}\cup S^{(2)}$ over $\partial \mathcal{F} \times \mathcal{A}$.
\end{proposition}
\begin{proof}
    In this proof, we use the notation introduced in \cref{def::algorithm_sequences,def::V_R}.

    By \itemcref{prp::analysis}{prp::analysis_iii}, the sequence $\bigl((x^k_1, x^k_2, z^k)\bigr)_{k \in \N}$ is bounded, therefore the sequence $\bigl((x^k_1, x^k_2, z^k)\bigr)_{k \in \N}$ has a nonempty set of weak sequential cluster points, see \cite[Lemma~2.45]{BauschkeCombettes2017}. 
    
    We begin by showing that all weak sequential cluster points of $\bigl((x^k_1, x^k_2, z^k)\bigr)_{k \in \N}$ are elements of $\Fix(\partial f, B, \alpha, \beta)$. 
    To that end, suppose that $(\bar{x}_1, \bar{x}_2, \bar{z}) \in \Hil^3$ is a weak sequential cluster point of $\bigl((x^k_1, x^k_2, z^k)\bigr)_{k \in \N}$, say $(x^{k_n}_1, x^{k_n}_2, z^{k_n}) \weak (\bar{x}_1, \bar{x}_2, \bar{z})$ as $n \to \infty$. By \itemcref{prp::analysis}{prp::analysis_ii}, we get that $\bar{x}_1 = \bar{x}_2$. Let $\bar{x} \coloneqq \bar{x}_1$. With the definition of $u_2^k$ in \cref{def::algorithm_sequences}, we get that the subsequence $(x^{k_n}_1, u^{k_n}_2) \weak \left(\bar{x}, \frac{1}{\alpha}\left(\bar{x} - \bar{z}\right) \right)$ as $n \to \infty$.
    Moreover \begin{align*}
        \begin{cases}
        x^k_1 &= J_{\alpha \partial f}(z^k) \\
        x^k_2 &= J_{\beta B}\left(\left(1+\frac{\beta}{\alpha}\right)x^k_1 - \frac{\beta}{\alpha}z^k\right)
    \end{cases} &\iff \begin{cases}
        u^k_1 &\in \partial f(x^k_1) \\
        u^k_2 &\in B(x^k_2)
    \end{cases} \\
    &\iff \begin{cases}
        u^k_1 + u^k_2 &\in \partial f(x^k_1) + u^k_2 \\
        x^k_2 - x^k_1  &\in - x^k_1 + B^{-1}(u^k_2)
    \end{cases} \\
    &\iff \begin{bmatrix}
        u^k_1 + u^k_2 \\
        x^k_2 - x^k_1
    \end{bmatrix} \in \underbrace{\begin{bmatrix}
        \partial f & \Id \\
        -\Id & B^{-1}
    \end{bmatrix}}_{=: C} \begin{bmatrix}
        x^k_1 \\
        u^k_2
    \end{bmatrix},
    \end{align*}  
    where from \itemcref{prp::analysis}{prp::analysis_ii} we conclude that $(u^k_1 + u^k_2, x^k_2 - x^k_1) \to 0$ as $k \to \infty$. Since $C$ equals the sum of two maximally monotone operators $$C = \begin{bmatrix}
        \partial f & 0 \\
        0 & B^{-1}
    \end{bmatrix} + \begin{bmatrix}
        0 & \Id \\
        -\Id & 0
    \end{bmatrix},$$ the latter of which has full domain, we conclude that $C$ is maximally monotone \cite[Corollary~25.5(i)]{BauschkeCombettes2017} and that $\gra(C)$ is sequentially closed in the weak-strong topology of $\Hil \times \Hil$ \cite[Proposition~20.38]{BauschkeCombettes2017}, implying, since $(u^{k_n}_1 + u^{k_n}_2, x^{k_n}_2 - x^{k_n}_1)\in C(x_1^{k_n},u_2^{k_n})$, that $\left(\bar{x}, \frac{1}{\alpha}\left(\bar{x} - \bar{z}\right) \right) \in \zer C$. This implies that
    \begin{align*}
            \begin{bmatrix}
                0 \\
                0 
            \end{bmatrix} \in \begin{bmatrix}
                \partial f & \Id \\
                -\Id & B^{-1} 
            \end{bmatrix} \begin{bmatrix}
                \bar{x} \\
                \frac{1}{\alpha}(\bar{x} - \bar{z})
            \end{bmatrix} &\iff
            \begin{cases}
                0 &\in \partial f(\bar{x}) + \frac{1}{\alpha}(\bar{x} - \bar{z}) \\
                0 &\in - \bar{x} + B^{-1}\left(\frac{1}{\alpha}(\bar{x} - \bar{z})\right)
            \end{cases} \\
            &\iff 
            \begin{cases}
            \frac{1}{\alpha}(\bar{z} - \bar{x}) &\in \partial f(\bar{x}) \\
            \frac{1}{\alpha}(\bar{x} - \bar{z}) &\in B(\bar{x})
            \end{cases} \\
            &\iff 
            \begin{cases}
            \frac{1}{\alpha}(\bar{z} - \bar{x}) &\in \partial f(\bar{x}) \\
            \frac{1}{\beta}\left(\left(1 + \frac{\beta}{\alpha} \right)\bar{x} - \frac{\beta}{\alpha}\bar{z} - \bar{x}\right)  &\in B(\bar{x})
            \end{cases} \\
             &\iff 
            \begin{cases}
            \bar{x} &= J_{\alpha \partial f}(\bar{z}) \\
            \bar{x} &= J_{\beta B}\left(\left(1 + \frac{\beta}{\alpha} \right)\bar{x} - \frac{\beta}{\alpha}\bar{z} \right) 
            \end{cases} \\
            &\iff (\bar{x}, \bar{x}, \bar{z}) \in \Fix(\partial f, B, \alpha, \beta).
    \end{align*} 
    We have shown that all weak sequential cluster points of $\bigl((x^k_1, x^k_2, z^k)\bigr)_{k \in \N}$ are elements of $\Fix(\partial f, B, \alpha, \beta)$. 

    Next, we show that there exists only one weak sequential cluster point of the sequence $\bigl((x^k_1, x^k_2, z^k)\bigr)_{k \in \N}$. Let $v^k \coloneqq (x^k_1, x^k_2, z^k)$ for all $k \in \N$. Suppose that $\bar{v} \coloneqq (\bar{x}_1, \bar{x}_2, \bar{z}) \in \Hil^3$ and $\hat{v} \coloneqq (\hat{x}_1, \hat{x}_2, \hat{z}) \in \Hil^3$ are two weak sequential cluster points of $(v^k)_{k \in \N}$, say $v^{k_n} \weak \bar{v}$ and $v^{\ell_n} \weak \hat{v}$ as $n \to \infty$. We get from before that $\bar{x}_1 = \bar{x}_2$ and $\hat{x}_1 = \hat{x}_2$. Define $\bar{x} \coloneqq \bar{x}_1$ and $\hat{x} \coloneqq \hat{x}_1$. Moreover, both $\bar{v}$ and $\hat{v}$ lie in $\Fix(\partial f, B, \alpha, \beta)$. 

    Let $i \in \{1, 2\}$ be such that $(\alpha, \beta, \theta) \in S^{(i)}$ and let $\bar{Q}_V^{(i)} \in \Sym^{3}$ be the matrix from \cref{rem::quadratic_form_remark}. By \itemcref{prp::analysis}{prp::analysis_iv}, both $\left( \mathcal{Q}\left(\bar{Q}_V^{(i)}, v^k - \bar{v}\right)\right)_{k \in \N}$ and $\left( \mathcal{Q}\left(\bar{Q}_V^{(i)}, v^k - \hat{v}\right)\right)_{k \in \N}$ converge. Since 
    $$2 \left\langle v^k, \bar{Q}_V^{(i)}(\bar{v} - \hat{v}) \right\rangle =  \mathcal{Q}\left(\bar{Q}_V^{(i)}, v^k - \bar{v}\right) -  \mathcal{Q}\left(\bar{Q}_V^{(i)}, v^k - \hat{v}\right) +  \mathcal{Q}\left(\bar{Q}_V^{(i)}, \bar{v}\right) + \mathcal{Q}\left(\bar{Q}_V^{(i)}, \hat{v}\right)$$ we get that $\left(\left\langle v^k, \bar{Q}_V^{(i)}\left(\bar{v} - \hat{v}\right) \right\rangle\right)_{k \in \N}$ converges as well. Therefore, passing along the subsequences $(v^{k_n})_{n \in \N}$ and $(v^{\ell_n})_{n \in \N}$ gives that $$\left\langle \bar{v}, \bar{Q}_V^{(i)}(\bar{v} - \hat{v}) \right\rangle = \left\langle \hat{v}, \bar{Q}_V^{(i)}(\bar{v} - \hat{v}) \right\rangle,$$ which implies that \begin{equation} \label{eq::fundamental_equation}
        \mathcal{Q}\left(\bar{Q}_V^{(i)}, \bar{v} - \hat{v}\right) = 0.
    \end{equation}
    If $i = 1$, then \eqref{eq::fundamental_equation} equals $$\norm{\bar{z} - \hat{z} + (\tau - 1)(\bar{x} - \hat{x})}^2 + \tau(1 - \tau)\norm{\bar{x} - \hat{x}}^2 = 0,$$
    where $(1-\tau)$ is nonnegative for $(\alpha, \beta, \theta)\in S^{(1)}$.
    If $i = 2$, then \eqref{eq::fundamental_equation} equals \begin{equation*}
        \norm{\bar{z} - \hat{z}}^2 + \frac{\theta(4\tau - \theta)}{4}\norm{\frac{2(\tau - 1)}{4\tau - \theta}(\bar{x} - \hat{x})}^2 + \frac{\tau(\tau - 1)(4 - \theta)}{4\tau - \theta}\norm{\bar{x} - \hat{x}}^2 = 0,
    \end{equation*}
    where $(4\tau-\theta)$, $(4-\theta)$ and $(\tau-1)$ are nonnegative for $(\alpha, \beta, \theta)\in S^{(2)}$.
    
    This implies that $\bar{z} = \hat{z}$ in both cases. Since $$\bar{x} = J_{\alpha \partial f}(\bar{z}) = J_{\alpha \partial f}(\hat{z}) =\hat{x}$$ we find that $\bar{v} = \hat{v}$, proving uniqueness of weak sequential cluster points.
    
    Therefore, $\bigl((x^k_1, x^k_2, z^k)\bigr)_{k \in \N}$ is bounded and every weak sequential cluster point of $\bigl((x^k_1, x^k_2, z^k)\bigr)_{k \in \N}$ belongs to $\Fix(\partial f, B, \alpha, \beta)$, and so by \cite[Lemma~2.46]{BauschkeCombettes2017}, the sequence $\bigl((x^k_1, x^k_2, z^k)\bigr)_{k \in \N}$ converges weakly, and since all weak sequential cluster points lie in $\Fix(\partial f, B, \alpha, \beta)$, we have that $\bigl((x^k_1, x^k_2, z^k)\bigr)_{k \in \N}$ converges weakly to a point in $\Fix(\partial f, B, \alpha, \beta)$.
    \end{proof}

\subsection{Counterexamples}\label{sec::Counter_Examples}
In the previous section, we have shown weak convergence of \cref{alg::main} over the parameter region $S^{(1)}\cup S^{(2)}$. With the following counterexamples we show that this region is tight, in the sense that \cref{alg::main} fails to be unconditionally convergent over $\partial \mathcal{F}\times \partial \mathcal{F}$ for parameters $(\alpha, \beta, \theta)$ outside of $S^{(1)}\cup S^{(2)}$.
\begin{proposition}\label{prp::counters}
    If \cref{alg::main} is unconditionally convergent for $(\alpha, \beta, \theta) \in \R_{++}^2 \times (\R \setminus \{0\})$ over $\partial\mathcal{F}\times \partial\mathcal{F}$ then $(\alpha, \beta, \theta)\in S^{(1)}\cup S^{(2)}$, i.e., $S(\partial\mathcal{F}\times\partial\mathcal{F})\subset S^{(1)}\cup S^{(2)}$.
\end{proposition}
\begin{proof}    
    Let $(\alpha, \beta, \theta) \in \R_{++}^2 \times (\R \setminus \{0\})$ be such that \cref{alg::main} is unconditionally convergent for $(\alpha, \beta, \theta)$ over $\partial\mathcal{F}\times \partial\mathcal{F}$.
    
    First, let $f = 0 \in \mathcal{F}$ and $g = \iota_{\{0\}} \in \mathcal{F}$. Note that $\zer(\partial f + \partial g) = \{0\}$. Let $\bigl((x^k_1, x^k_2, z^k)\bigr)_{k \in \N}$ be generated by \cref{alg::main} applied to $(\partial f, \partial g, z^0, \alpha, \beta, \theta)$ given some initial point $z^0 \in \Hil \setminus \{0\}$. Note that such an element $z^0$ exists since $\dim \Hil \geq 1$. Then
    \begin{align*}
        x_1^k &= \prox_{\alpha f}(z^k) = z^k\\
        x_2^k &= \prox_{\beta g}\left(\left(1+\frac{\beta}{\alpha}\right)x_1^k-\frac{\beta}{\alpha}z^k\right) = 0\\
        z^{k+1} &= z^k+ \theta(x_2^k-x_1^k) = (1-\theta)z^k
    \end{align*}
    for every $k\in \N$. We conclude in particular that $\theta \in (0,2)$.

    Likewise, let us repeat the above argument for $f = \iota_{\{0\}} \in \mathcal{F}$ and $g = 0 \in \mathcal{F}$ where again $\zer(\partial f + \partial g) = \{0\}$. Then
        \begin{align*}
        x_1^k &= \prox_{\alpha f}(z^k) = 0\\
        x_2^k &= \prox_{\beta g}\left(\left(1+\frac{\beta}{\alpha}\right)x_1^k-\frac{\beta}{\alpha}z^k\right) = -\frac{\beta}{\alpha}z^k \\
        z^{k+1} &= z^k+ \theta(x_2^k-x_1^k) = \left(1-\theta\frac{\beta}{\alpha}\right)z^k
    \end{align*} for every $k\in \N$. We conclude in particular that $\theta \in (0, 2\alpha/\beta)$.
    
    Put together, if \cref{alg::main} is unconditionally convergent for $(\alpha, \beta, \theta) \in \R_{++}^2 \times (\R \setminus \{0\})$ over $\partial\mathcal{F}\times \partial\mathcal{F}$ then $\theta\in (0,2)\cap (0, 2\alpha/\beta)$, i.e., $(\alpha, \beta, \theta)\in S^{(1)}\cup S^{(2)}.$
\end{proof}

\subsection[]{Proof of \cref{thm::main_theorem}}\label{sec::thm_proof}
Combining the results in \cref{sec::Lyapunov_analysis,sec::Counter_Examples} we can now prove our main result.
\begin{proof}[Proof of \cref{thm::main_theorem}]
   \cref{prp::weak_convergence} gives that $$S^{(1)} \cup S^{(2)} \subset S(\partial\mathcal{F}\times \mathcal{A}) \subset S(\partial\mathcal{F}\times \partial\mathcal{F}).$$ \cref{prp::counters} gives that $S(\partial\mathcal{F}\times\partial\mathcal{F})\subset S^{(1)}\cup S^{(2)}$. Therefore,
   $$S^{(1)} \cup S^{(2)} = S(\partial\mathcal{F}\times \mathcal{A}) = S(\partial\mathcal{F}\times \partial\mathcal{F}).$$
\end{proof}
\section{Convergence Rates}
\label{sec::Additional_convergence_results}
In this section, we show ergodic convergence of the primal-dual gap for the extended Douglas--Rachford method in \cref{alg::main}, with convergence rate $\mathcal{O}(1/K)$.
We consider inclusion problems of the form
$$
\text{find } x\in \Hil \text{ such that } 0\in \partial f(x) + \partial g(x),
$$ where $f, g : \Hil\rightarrow \R\cup\{+\infty\} $ are proper, closed, and convex, such that $\zer(\partial f + \partial g) \neq \emptyset$. Let $(x^\star, u^\star)\in \Hil^2$ be such that
\begin{align}\label{eq::ustar}
    \begin{cases}
        u^\star&\in \partial f(x^\star)\\  
        x^\star &\in\partial g^*(-u^\star) 
    \end{cases}.
\end{align} 
We define the Lagrangian functions $\mathcal{L}\colon \mathcal{H}^2\rightarrow \R\cup\{-\infty,+\infty\}$ and $\bar{\mathcal{L}}\colon \mathcal{H}^2\rightarrow \R\cup\{-\infty,+\infty\}$ as
\begin{align*}
    \mathcal{L}(x,u) &\coloneqq f(x) + \langle x,u\rangle-g^*(u),\\
    \bar{\mathcal{L}}(x,u) &\coloneqq g(x) + \langle x,u\rangle-f^*(u).
\end{align*} Whenever the expression $+\infty - (+\infty)$ arises, it is understood to take the value $-\infty$. This convention is adopted for definiteness, although this indeterminate case will not occur in what follows. We define the primal-dual gap functions 
$\mathcal{D}_{x^\star, -u^\star}: \mathcal{H}^2\rightarrow \R\cup \{+\infty\}$ and $\bar{\mathcal{D}}_{x^\star, u^\star}: \mathcal{H}^2\rightarrow \R\cup \{+\infty\}$ as
\begin{align*}
    \mathcal{D}_{x^\star, -u^\star}(x, u) &\coloneqq \mathcal{L}(x, -u^\star)- \mathcal{L}(x^\star, u)\\
    \bar{\mathcal{D}}_{x^\star, u^\star}(x,u) &\coloneqq \bar{\mathcal{L}}(x,u^\star)-\bar{\mathcal{L}}(x^\star,u),
\end{align*}
for all $(x,u)\in \Hil^2$. In particular
\begin{equation}
    \label{eq::PD-gap}
    \begin{aligned}
    \mathcal{D}_{x^\star, -u^\star}(x,u) &= f(x) - \langle x, u^\star \rangle - f(x^\star) + g^*(u) - \langle x^\star, u\rangle -g^*(-u^\star)\\
    &= f(x) - \langle x-x^\star, u^\star \rangle - f(x^\star) + g^*(u) - \langle x^\star, u+u^\star\rangle -g^*(-u^\star)\\
    &= D_f(x,x^\star, u^\star) + D_{g^*}(u, -u^\star, x^\star),
    \end{aligned}
\end{equation}
and similarly
\begin{equation} 
    \label{eq::PD-gap_bar}
    \bar{\mathcal{D}}_{x^\star, u^\star}(x,u) = D_{f^*}(u, u^\star, x^\star) + D_g(x,x^\star, -u^\star),
\end{equation}
for all $(x,u)\in \mathcal{H}^2$. From the convexity of $f$ and $g$, \cref{eq::ustar}, and \cref{lem::Dh_nonnegative}, it follows that
$\mathcal{D}_{x^\star, -u^\star}$ and $\bar{\mathcal{D}}_{x^\star, u^\star}$ are convex and nonnegative.

Before presenting the ergodic convergence rate result, let us define the ergodic sequences of \cref{alg::main}.

\begin{definition}[Ergodic iterates of \cref{alg::main}]
    \label{def::Ergodic_iterates}
    Let $f,g$ be proper, closed, and convex functions, let $z^0\in \Hil$, and let $(\alpha, \beta, \theta) \in S^{(1)}\cup S^{(2)}$. For $\bigl((x^k_1, x^k_2, z^k)\bigr)_{k \in \N}$ generated by \cref{alg::main} applied to $(\partial f, \partial g, z^0, \alpha, \beta, \theta)$ with $(u_1^k)_{k\in \N}$ and $(u_2^k)_{k\in \N}$ as in \cref{def::algorithm_sequences}, we define the ergodic iterates
\begin{align*} 
        \bar{x}_j^K &\coloneqq \frac{1}{K+1}\sum_{k=0}^{K} x_j^k
\qquad\quad\text{and}\qquad\quad        \bar{u}_j^K \coloneqq \frac{1}{K+1}\sum_{k=0}^{K} u_j^k
\end{align*}
for each $j\in \{1,2\}$ and $K\in \N$. 
\end{definition}
\begin{proposition}
Following \cref{def::Ergodic_iterates}, let $(x^\star, u^\star)\in \Hil^2$ satisfy \cref{eq::ustar}. Then the primal-dual gap sequences 
$$\bigl(\mathcal{D}_{x^\star, -u^\star}(\bar x_1^K, \bar u_2^K)\bigr)_{K\in \N}\quad\text{and}\quad\bigl(\bar{\mathcal{D}}_{x^\star, u^\star}(\bar x_2^K, \bar u_1^K)\bigr)_{K\in \N}$$ converge to zero with rate $\mathcal{O}(1/K)$.
\end{proposition}
\begin{proof}
We will use the following identity. If $h:\Hil\to\R\cup\{+\infty\}$ is proper, closed, and convex, $(x,y)\in \Hil^2$, $u\in \partial h(x)$, and $v\in \partial h(y)$, then 
\begin{equation}
\begin{aligned} 
    \label{eq::Bregman_conjugate}
    D_h(y, x, u) + D_{h^*}(v, u,x) &= h(y)-h(x)-\langle u, y-x\rangle + h^*(v) - h^*(u) - \langle x, v-u\rangle\\
    &= \left(h(y) + h^*(v) - \langle v, y\rangle \right) + \langle v, y\rangle \\
    &\quad-\left(h(x) + h^*(u) - \langle x,u\rangle\right) -\langle u,y\rangle - \langle x, v-u\rangle\\
    &= \langle v-u, y-x\rangle,
\end{aligned}
\end{equation}
since the expressions in the parentheses are zero by the Fenchel--Young equality.

By \cref{prp::fix_point_prp}, there exists some $z^\star \in \Hil$ such that $(x^\star, x^\star, z^\star)\in \Fix(\partial f, \partial g,\alpha, \beta)$. 
Combining \cref{eq::PD-gap} with \cref{eq::PD-gap_bar}, and then using \cref{eq::Bregman_conjugate} for $f$ and $g$, we get that
\begin{equation}
\label{eq::Bregman_equality_Duality_gap}
\begin{aligned}
     \mathcal{D}_{x^\star, -u^\star}(x_1^k, u_2^k) + \bar{\mathcal{D}}_{x^\star, u^\star}(x_2^k, u_1^k) &= D_f(x_1^k, x^\star, u^\star) + D_{f^*}(u_1^k, u^\star, x^\star) \\
     &\quad + D_g(x_2^k, x^\star, -u^\star)
    + D_{g^*}(u_2^k,-u^\star, x^\star)\\
     &= \langle u_1^k-u^\star, x_1^k-x^\star \rangle + \langle u_2^k+u^\star,x_2^k-x^\star \rangle
\end{aligned}
\end{equation}
 for all $k\in \N$.

Since $\mathcal{D}_{x^\star, -u^\star}$ and $\bar{\mathcal{D}}_{x^\star, u^\star}$ are convex, we have by Jensen's inequality that
\begin{align*}
    \mathcal{D}_{x^\star, -u^\star}(\bar x_1^K, \bar u_2^K)  &+ \bar{\mathcal{D}}_{x^\star, u^\star}(\bar x_2^K, \bar u_1^K) \\
    &\leq \frac{1}{K+1}\sum_{k = 0}^{K}\left(\mathcal{D}_{x^\star, -u^\star}(x_1^k, u_2^k) + \bar{\mathcal{D}}_{x^\star, u^\star}(x_2^k, u_1^k)\right) \\
    &= \frac{1}{K+1}\sum_{k=0}^{K}\left(\langle u_1^k-u^\star, x_1^k-x^\star\rangle + \langle u_2^k+u^\star, x_2^k-x^\star\rangle\right) 
    \\
    &= \frac{1}{K+1}\sum_{k=0}^{K}\left(I_k - \langle u_1^{k+1}-u^\star, x_1^{k+1}- x^\star\rangle - \langle u_2^{k+1}+u^\star, x_2^{k+1}- x^\star\rangle\right)  \\
    &\leq \frac{1}{K+1} \sum_{k=0}^{K} I_k\leq \frac{1}{\alpha\theta (K+1)}\sum_{k=0}^{K}(\alpha \theta I_k + R_k)
    \leq \frac{V_0}{\alpha\theta (K+1)}
\end{align*}
for all $K\in \N$ and for $V_k$ and $R_k$ defined in \cref{def::V_R} and $I_k$ defined in \cref{def::I_k}, where the first equality follows from \cref{eq::Bregman_equality_Duality_gap}, the second equality from \cref{def::I_k}, the second inequality from monotonicity of $\partial f$ and $\partial g$ since $(x_1^{k+1},u_1^{k+1}),(x^\star,u^\star)\in\gra\partial f$ and $(x_2^{k+1},u_2^{k+1}),(x^\star,-u^\star)\in\gra\partial g$ by \cref{rem::subdifferential}, the third inequality from nonnegativity of $R_k$ (\cref{prp::analysis}), and the fourth inequality by summation and telescoping of the Lyapunov equality \eqref{eq::lemma_equality} in \cref{lemma::analysis} and by nonnegativity of $V_k$ (\cref{prp::analysis}).
By nonnegativity of the primal-dual gap functions, this shows ergodic convergence with rate $\mathcal{O}(1/K)$.
\end{proof}

\section{Extending ADMM, Chambolle--Pock, and more} \label{sec::admm_cp}

In this section, we apply \cref{alg::main} to different problem formulations and, by invoking \cref{thm::main_theorem}, arrive at new convergent variants of ADMM, the Chambolle--Pock method, and Douglas--Rachford splitting applied to a consensus problem formulation, sometimes referred to as parallel splitting. This section concludes with the derivation of a new version of the Generalized Alternating Projections method, allowing for overreflected projections. 

\subsection{Extending ADMM} \label{sec::admm}
To arrive at an extended version of ADMM, we follow the derivation and notation of \cite[Section~3.1]{ryu2022large}. We specialize to the finite-dimensional case $\Hil = \R^n$, though we note that, under mild assumptions, we also have weak convergence in the general real Hilbert setting. Let $f : \R^p \to \R \cup \{+\infty\}$ and $g : \R^q \to \R \cup \{+\infty\}$ be proper, closed, and convex functions. Let $A \in \R^{n \times p}, B \in \R^{n \times q}$, and $c \in \R^n$. Consider the following convex optimization problem: 
\begin{equation} \label{eq::Problem_ADMM}
\begin{aligned}
    \minimize_{(x, y)\in \R^p \times \R^q}\; &f(x) + g(y)\\
    \text{subject to } &Ax + By = c
\end{aligned}
\end{equation}
which can be reformulated as the following unconstrained problem,
\begin{align} \label{eq::Problem_ADMM_2} 
    \minimize_{z\in \R^n}\; (A\rhd f)(z) + (B\rhd g)(c-z),
\end{align}
using the infimal postcomposition, defined as
$$
(A\rhd f)(z) \coloneqq \inf \{f(x) \mid Ax = z\}
$$
for each $z\in\R^n$ and analogously for $B \rhd g$.
The classical ADMM method alternately minimizes
the augmented Lagrangian of \cref{eq::Problem_ADMM}
$$L_\alpha(x,y,u) := f(x) + g(y) + \langle u, Ax + By-c\rangle + \frac{\alpha}{2}\|Ax + By-c\|^2$$ 
and updates the dual variable $u$ according to
\begin{equation}
\label{eq::Classical_ADMM}
\left\lfloor
\begin{aligned}
        x^{k+1}&\in \Argmin_{x\in \R^p}L_{\alpha}\left(x, y^k, u^k\right)\\
        y^{k+1} &\in\Argmin_{y\in \R^q}L_\alpha ( x^{k+1}, y, u^k)\\
        u^{k+1} &= u^k + \alpha(Ax^{k+1}+By^{k+1}-c)
\end{aligned} \right. .
\end{equation}
This method can be derived by applying unrelaxed Douglas--Rachford splitting (i.e., with $\theta=1$) either to \cref{eq::Problem_ADMM_2} or to the corresponding dual problem
\begin{equation}  
    \label{eq::ADMM_Dual}
    \minimize_{u\in \Hil}\underbrace{f^*(-A^\top u)}_{=: \tilde f(u)} + \underbrace{g^*(-B^\top u) + c^\top u}_{=: \tilde g(u)}.
\end{equation}

By instead applying \cref{alg::main} to \cref{eq::ADMM_Dual},  and invoking the convergence result in \cref{thm::main_theorem}, we obtain new convergent variants of ADMM.
For a fixed relaxation parameter $\theta$, we arrive at \cref{alg::ADMM_simpler}, which extends the classical ADMM in \cref{eq::Classical_ADMM} by allowing for different penalty parameters $\alpha$ and $\beta$ in the two augmented Lagrangians. From \cref{thm::Full_characterization}, we conclude, under mild additional assumptions, that \cref{alg::ADMM_simpler} converges if $0< \alpha < 2\beta$. Using \cref{alg::ADMM_simpler} it is also possible to derive new versions of proximal ADMM (see for example \cite{ryu2022large} for a derivation with regular ADMM), and the preconditioned version of ADMM described in \cite[Section 3.4.2]{BoydParikhChuPeleatoEckstein2011}, with different penalty parameters $\alpha$ and $\beta$. 
\begin{algorithm}[htbp]
	\caption{Extended ADMM}
	\begin{algorithmic}[1]
	    \State \textbf{Input:} $(f, g, A, B)$ as in \eqref{eq::Problem_ADMM}, $u^0 \in \R^n$, $(\alpha, \beta) \in  \R_{++}^2$ such that $\alpha < 2\beta$.
		\For {$k=0,1,2,\ldots$}
		    \State \textbf{Update: } \begin{equation*}
                \left\lfloor
                \begin{aligned}
                    x^{k+1}&\in \Argmin_{x\in \R^p}L_{\beta}\left(x, y^k, u^k\right)\\
                    y^{k+1} &\in\Argmin_{y\in \R^q}L_\alpha ( x^{k+1}, y, u^k)\\
                    u^{k+1} &= u^k + \alpha(Ax^{k+1}+By^{k+1}-c)
                \end{aligned} \right. 
                \end{equation*}
		\EndFor
	\end{algorithmic}
    \label{alg::ADMM_simpler}
\end{algorithm}

For a general relaxation parameter $\theta\in(0,\min\{2,2\beta/\alpha\})$ we instead arrive at the following extended version of the generalized ADMM \cite{EcksteinBertsekas1992}, see \cref{alg::ADMM_extended}. 
\begin{algorithm}[htbp]
	\caption{Extended Generalized ADMM}
	\begin{algorithmic}[1]
    	\State \textbf{Input:} $(f, g, A, B)$ as in \eqref{eq::Problem_ADMM}, $u^0 \in \R^n$, $(\alpha, \beta, \theta) \in  \R_{++}^3$ such that $\theta < \min\{2, 2\beta/\alpha\}$.
		\For {$k=0,1,2,\ldots$}
		    \State \textbf{Update: } \begin{equation*}
                \left\lfloor
                \begin{aligned}
                    x^{k+1}&\in \Argmin_{x\in \R^p}L_{\beta}\left(x, y^k, u^k + \alpha(1-\theta)(By^k-c)\right)\\
                    y^{k+1} &\in\Argmin_{y\in \R^q}L_{\alpha} (\theta x^{k+1}, y, u^k)\\
                    u^{k+1} &= u^k + \theta\alpha(Ax^{k+1}+By^{k+1}-c)
                \end{aligned} \right. 
                \end{equation*}
            \EndFor
	\end{algorithmic}
\label{alg::ADMM_extended}
\end{algorithm}

For the derivation we use the following identity (see \cite[Equation~(2.6)]{ryu2022large}).
Let $h: \R^{m}\rightarrow \R\cup\{+\infty\}$ be a proper, closed and convex function and let $C\in \R^{n\times m}$. If $\relint\dom h^*\cap\ran (-C^\top) \neq \emptyset$ then
\begin{equation}
    \label{eq::Argmin_prox}
    v = \prox_{\alpha h^*\circ (-C^\top)}(u) \iff \begin{cases}
    \displaystyle x\in \Argmin_{x\in \R^n}\left(h(x)+\langle u, Cx \rangle + \frac{\alpha}{2}\norm{Cx}^2\right)\\
    v = u+\alpha Cx.
\end{cases}
\end{equation}

Applying the extended Douglas--Rachford method in \cref{alg::main} to \eqref{eq::ADMM_Dual} with parameters $(\alpha, \beta, \theta\alpha/\beta)$ for $(\alpha, \beta, \theta)\in \R_{++}^3$ 
gives the update
\begin{equation}
\label{eq::ADMM_DR_update}
\begin{aligned}
    \mu^{k+1/2} &= J_{\alpha \partial \tilde g}(\psi^k)\\
    \mu^{k+1} &= J_{\beta \partial \tilde f}\left(\left(1+\tfrac{\beta}{\alpha}\right)\mu^{k+1/2}-\tfrac{\beta}{\alpha}\psi^k\right)\\
    \psi^{k+1} &= \psi^k + \frac{\theta\alpha}{\beta}(\mu^{k+1}-\mu^{k+1/2}).
\end{aligned}
\end{equation}

 We have that $J_{\alpha \partial \tilde g}(\cdot) = J_{\alpha \partial (g^*\circ (-B^\top))}(\cdot-\alpha c)$ by \cite[Proposition~23.17]{BauschkeCombettes2017}. Then by using \cref{eq::Argmin_prox} and defining auxiliary variables $\tilde y^{k+1}\in \R^q$ and $\tilde{x}^{k+1}\in\R^p$, the update of \cref{eq::ADMM_DR_update} becomes
\begin{align*}
    \tilde{y}^{k+1} &\in \Argmin_{y\in \R^q} \left(g(y) + \left\langle \psi^k - \alpha c, By \right\rangle + \frac{\alpha}{2}\|By\|^2 \right)\\
    &= \Argmin_{y\in \R^q}\left(g(y) + \left\langle \psi^k - \theta\alpha A\tilde{x}^k, By\right\rangle + \frac{\alpha}{2}\|\theta A\tilde{x}^k + By-c \|^2\right)\\
    \mu^{k+1/2} &= \psi^k + \alpha B\tilde{y}^{k+1}-\alpha c\\
    \tilde{x}^{k+1}&\in \Argmin_{x\in \R^p}\left(f(x)+ \left\langle \left(1+\frac{\beta}{\alpha}\right)\mu^{k+1/2}-\frac{\beta}{\alpha}\psi^k, Ax \right\rangle + \frac{\beta}{2}\|Ax\|^2\right) \\
    &= \Argmin_{x\in \R^p}\left(f(x) + \left\langle \psi^k + (\alpha + \beta)(B\tilde{y}^{k+1}-c), Ax \right\rangle + \frac{\beta}{2}\|Ax\|^2\right)\\
    &= \Argmin_{x\in \R^p}\left(f(x) + \left\langle \psi^k + \alpha(B\tilde{y}^{k+1}-c), Ax\right\rangle + \frac{\beta}{2}\|Ax + B\tilde{y}^{k+1}-c\|^2\right)\\
    \mu^{k+1} &= \psi^{k} + (\alpha+\beta)(B\tilde{y}^{k+1}-c)  + \beta A\tilde{x}^{k+1}\\
    \psi^{k+1} &= \psi^k + \frac{\theta\alpha}{\beta}(\mu^{k+1}-\mu^{k+1/2}) = \psi^{k} + \theta\alpha\left( A\tilde{x}^{k+1} + B\tilde{y}^{k+1}-c\right).
\end{align*}
We can now remove $\mu^{k+1}$ and $\mu^{k+1/2}$ as they are redundant, and make the substitution
$u^k = \psi^k -\theta\alpha A\tilde{x}^k$, to
arrive at the update
\begin{align*}
    \tilde{y}^{k+1}&\in \Argmin_{y\in \R^q}\left(g(y) + \langle u^k, By\rangle + \frac{\alpha}{2}\|\theta  A\tilde{x}^k + By-c \|^2\right)\\
    \tilde{x}^{k+1}&\in \Argmin_{x\in \R^p}\left(f(x) + \langle u^{k+1}+(1-\theta)\alpha(B\tilde{y}^{k+1}-c), Ax\rangle + \frac{\beta}{2}\|Ax + B\tilde{y}^{k+1}-c\|^2\right)\\
    u^{k+1} &= u^k + \theta\alpha\left(A\tilde{x}^k + B\tilde{y}^{k+1}-c\right).
\end{align*}
Using the definition of the augmented Lagrangian and reordering, we get the updates of \cref{alg::ADMM_extended}
and by \cref{thm::main_theorem} we have convergence if $\theta\alpha/\beta \in (0,\min\{2, 2\alpha/\beta\}) \iff \theta \in (0, \min\{2\beta/{\alpha}, 2\})$. We then get \cref{alg::ADMM_simpler} by letting $\theta = 1$.

\subsection{Extending Chambolle--Pock} 
\label{sec::extended_chambolle_pock}
To arrive at an extended version of the Chambolle--Pock method, we follow the derivation and notation of \cite{o2020equivalence}. We specialize to the finite-dimensional case $\Hil = \R^n$, though we note that, under mild assumptions, we also have weak convergence in the general real Hilbert setting. Let $f : \R^n \to \R \cup \{+\infty\}$ and $g : \R^m \to \R \cup \{+\infty\}$ be proper, closed, and convex functions and let $A \in \R^{m \times n}$. Consider the optimization problem
\begin{equation} \label{eq::optimization_problem_cp}
    \minimize_{x \in \R^n} \  f(x) + g(Ax).
\end{equation} 
This optimization problem can be solved with the Primal-Dual Hybrid Gradient algorithm (PDHG), also known as the Chambolle--Pock algorithm \cite{chambolle2011first}, which is given by the update step
\begin{equation}
\left\lfloor
\begin{aligned}
x^{k+1} &= \prox_{\tau f}\left(x^{k} - \tau A^\top z^{k} \right)\\
z^{k+1} &= \prox_{\sigma g^*}\left(z^{k} + \sigma A\left(x^{k+1} + \theta(x^{k+1}- x^{k}) \right) \right)\\
\end{aligned} \right..\label{eq::CP_iteration}
\end{equation} 
This algorithm was originally shown to converge for $\theta = 1$ and $\tau\sigma \norm{A}^2< 1$, see \cite{chambolle2011first}, where $\norm{A}$ denotes the spectral norm of $A$.
Recent analyses of this Chambolle--Pock algorithm allow for larger parameter regions than in \cite{chambolle2011first}. 
In \cite{li2021new,ma2023understanding,yan2024improved}, the step-size convergence region is extended to $\tau \sigma \norm{A}^2 < 4/3$, still with $\theta=1$.
These results have been generalized in \cite{banert2025chambolle} to the $\theta\neq 1$ setting. In particular, they show convergence whenever $\theta > 1/2$ and $\tau \sigma \norm{A}^2 < 4/(1 + 2\theta)$. Moreover, the bound for $\tau \sigma \norm{A}^2$ is tight, see \cite{banert2025chambolle}. 

In \cite[Algorithm 4]{he2012convergence}, they propose the following relaxation for the Chambolle--Pock method:
\begin{equation}
\left\lfloor
\begin{aligned}
\bar{x}^{k} &= \prox_{\tau f}\left(x^{k} - \tau A^\top z^{k} \right)\\
\bar{z}^{k} &= \prox_{\sigma g^*}\left(z^{k} + \sigma A\left(2\bar{x}^k - x^{k} \right) \right)\\
x^{k+1} &= x^{k} + \rho\left(\bar{x}^k - x^{k} \right)\\
z^{k+1} &= z^{k} + \rho\left(\bar{z}^k - z^{k} \right)
\end{aligned} \right.\label{eq::rerelaxed_CP}
\end{equation} 
and show convergence whenever $\tau \sigma\norm{A}^2 < 1$ and $\rho \in (0, 2)$, which was later extended in \cite{Condat2013} to $\tau \sigma\norm{A}^2 \leq 1$ and $\rho \in (0, 2)$.
This method has been shown in \cite{o2020equivalence} to be a special case of the relaxed Douglas--Rachford algorithm when applied to a specific problem formulation. This route results in the same convergent parameter region as in \cite{Condat2013}, i.e., $\tau \sigma\norm{A}^2 \leq 1$ and $\rho \in (0, 2)$.
 
When the extended Douglas--Rachford method proposed in this paper is combined with the approach of \cite{o2020equivalence}, we obtain the doubly relaxed Chambolle--Pock method of \cref{alg::extended_chambolle_pock}, combining both types of relaxations. See \cref{sec::CP_derivation} for its derivation. 

\begin{algorithm}[htbp]
	\caption{Doubly Relaxed Chambolle--Pock}
	\begin{algorithmic}[1]
        \State \textbf{Input:} $(f, g, A)$ as in \cref{eq::optimization_problem_cp}, $(x^0, z^0) \in \R^n \times \R^m$, \\ \hspace{12.5mm} $(\tau, \sigma, \theta, \rho) \in \R^4_{++}$ such that $\rho \in (0, \min\{2, 2\theta\})$ and $\tau\sigma\norm{A}^2 \leq 1/\theta$.
		\For {$k=0,1,2\ldots$}
		    \State \textbf{Update: }
            \begin{equation*}
            \left\lfloor
            \begin{aligned}
            \bar{x}^{k} &= \prox_{\tau f}\left(x^{k} - \tau A^\top z^{k} \right)\\
            \bar{z}^{k} &= \prox_{\sigma g^*}\left(z^{k} + \sigma A\left(\bar{x}^k + \theta (\bar{x}^k - x^{k}) \right) \right)\\
            x^{k+1} &= x^{k} + \rho\left(\bar{x}^k - x^{k} \right)\\
            z^{k+1} &= z^{k} + \rho\left(\bar{z}^k - z^{k} \right)
            \end{aligned} \right.
            \end{equation*}
		\EndFor
	\end{algorithmic}
\label{alg::extended_chambolle_pock}
\end{algorithm}

This algorithm is guaranteed to converge on a strictly larger parameter region---namely $(\tau, \sigma, \theta, \rho) \in \R^4_{++}$ such that $\rho < \min\{2, 2\theta\}$ and $\tau\sigma\norm{A}^2 \leq 1/\theta$---than that established in \cite{o2020equivalence}. In fact, our analysis of \cref{alg::extended_chambolle_pock} is tight in the sense that neither of these two conditions can be relaxed without restricting the other. 
To see this, let $n=m=1$, $A=1$, $f=g^*=0$.
This defines an optimization problem of the form \cref{eq::optimization_problem_cp} with the unique minimizer $0$,
and gives the iteration
\begin{align}
    \label{eq::Matrix_CP}
    v^{k+1} = \begin{bmatrix} 1 & -\rho\tau\\
    \rho\sigma & 1-\rho\sigma\tau(1+\theta)
    \end{bmatrix}v^k,
\end{align}
where $v^k=(x^k,z^k)$. With the variable $\omega \coloneqq \tau\sigma\norm{A}^2$, we let $\lambda_{-}$ denote the eigenvalue of \cref{eq::Matrix_CP} given by the expression
\begin{align*}
    \lambda_{-}(\rho, \omega, \theta) \coloneqq 1-\frac{\rho\omega(1+\theta)}{2} - \frac{\rho\sqrt{\omega}}{2}\sqrt{\omega(1+\theta)^2-4}
\end{align*} for each $(\rho, \omega, \theta) \in \R_{++}^3$. For all $\omega \geq 1/\theta$ with $\theta,\rho>0$,
the discriminant of $\lambda_{-}(\rho, \omega, \theta)$ is nonnegative and $\lambda_{-}(\rho, \omega, \theta) \in \R$. Note that $\lambda_-(\rho, \omega, \theta)$ is continuous and decreasing in the variables $\rho, \omega$ for $\omega\geq 1/\theta$, i.e., $\lambda_-(\bar{\rho}, \bar{\omega}, \theta) \leq \lambda_-(\rho, \omega, \theta)$ holds for all $\bar{\rho} \geq \rho$ and $\bar{\omega} \geq \omega\geq 1/\theta$. At the bound $\rho = \min\{2, 2\theta\}$ and $\omega = 1/\theta$ we have after a short calculation that $\lambda_-(\min\{2, 2\theta\}, 1/\theta, \theta) = -1$ for all $\theta \in \R_{++}$, which gives, by continuity, that none of the upper bounds can be relaxed without simultaneously restricting the other.

This result reveals a trade-off between the two bounds: the conditions $\rho < \min\{2, 2\theta\}$ and $\tau\sigma\norm{A}^2 \leq 1/\theta$ cannot be relaxed independently---making one upper bound larger necessarily forces the other to become more restrictive. However, if we, for instance, fix $\rho\in(0,\min\{2, 2\theta\})$ to a specific value, the admissible step-size region can be enlarged. For example, for $\rho = 1$ and $\theta > 1/2$, the analysis in \cite{banert2025chambolle} allows $\tau\sigma\norm{A}^2 < 4/(1 + 2\theta)$, which is strictly less restrictive than our bound in this regime. It remains to be seen if the Chambolle--Pock method can be derived from the Douglas--Rachford method without conservatism in the step-size restriction for each individual choice of $\rho$.

Let us summarize our contributions in the doubly relaxed Chambolle--Pock setting. We present, to the best of the authors' knowledge, the first convergence guarantee for the doubly relaxed Chambolle--Pock method, i.e., with freedom in both relaxation parameters $\rho$ and $\theta$. Interestingly, the step sizes $\tau,\sigma$ can be chosen arbitrarily large by taking the relaxation parameters $\rho$ and $\theta$ small enough. We also provide the first convergence analysis for the case $\theta\in(0,1/2)$.

\subsubsection{Derivation of Doubly Relaxed Chambolle--Pock}
\label{sec::CP_derivation}

We will now derive \cref{alg::extended_chambolle_pock}, following the steps in \cite{o2020equivalence}, and prove its parameter convergence region. Let $\gamma>0$ be such that $\gamma \norm{A} \leq 1$. Let $C \coloneqq (\gamma^{-2}I - A A^\top)^{1/2} \in \R^{m \times m}$ and $B \coloneqq \begin{bmatrix} A & C \end{bmatrix} \in \R^{m \times (n+m)}$, implying that $B B^\top = \gamma^{-2}I$. The matrices $C$ and $B$ are introduced only for the derivation and the final algorithm will be independent of these, see \cref{alg::extended_chambolle_pock}. Consider the lifted proper, closed, and convex functions $\tilde{f}, \tilde{g} : \R^{n} \times \R^m \to \R \cup \{+\infty\}$ defined by
\begin{align*}
    \tilde{f}(u_1, u_2) &\coloneqq f(u_1) + \iota_{\{0\}}(u_2), \\
    \tilde{g}(u_1, u_2) &\coloneqq g(Au_1 + Cu_2),
\end{align*}for each $u = (u_1, u_2) \in \R^{n} \times \R^m$. Consider the optimization problem
\begin{equation} \label{eq::optimization_problem_cp_2}
    \minimize_{u \in \R^n\times\R^m} \  \tilde{f}(u) + \tilde{g}(u),
\end{equation} which is equivalent to \eqref{eq::optimization_problem_cp}. The corresponding optimality condition to \eqref{eq::optimization_problem_cp_2} is
\begin{equation} \label{eq::cp_optimality_condition_2}
    0 \in \partial \tilde{f}(u) + B^\top \circ \partial g\circ B(u),
\end{equation} 
where, under mild assumptions, $\partial \tilde g = B^\top \circ \partial g \circ B : \R^{n + m} \rightrightarrows \R^{n + m}$ and 
solving the inclusion problem \cref{eq::cp_optimality_condition_2} is equivalent to solving the optimization problem \cref{eq::optimization_problem_cp}, see \cite{rockafellar1970convex,BauschkeCombettes2017}.

Let $(\tau, \eta, \rho) \in S^{(1)} \cup S^{(2)}$ and $y^0 = (y^0_1, y^0_2) \in \R^n \times \R^m$, then \cref{alg::main} applied to \eqref{eq::cp_optimality_condition_2} gives
\begin{equation} \label{eq::cp_alg_1}
\begin{aligned}
    \bar{u}^k &= \prox_{\tau \tilde{f}}(y^k) \\
    w^k &= \prox_{\eta \tilde g}\left(\left(1 + \frac{\eta}{\tau} \right)\bar{u}^k - \frac{\eta}{\tau}y^k \right) \\
    y^{k+1} &= y^k + \rho\left(w^k - \bar{u}^k \right).
\end{aligned}
\end{equation} Using $\bar{w}^k \coloneqq \left(\left(1 + \frac{\eta}{\tau} \right)\bar{u}^k - \frac{\eta}{\tau}y^k \right) - w^k$ and the Moreau decomposition we get that \eqref{eq::cp_alg_1} becomes
\begin{equation}\label{eq::cp_alg_2}
\begin{aligned}
    \bar{u}^k &= \prox_{\tau \tilde{f}}(y^k) \\
    \bar{w}^k &= \prox_{(\eta \tilde{g})^{*}}\left(\left(1 + \frac{\eta}{\tau} \right)\bar{u}^k - \frac{\eta}{\tau}y^k \right) \\
    y^{k+1} &= \left( 1 - \rho \frac{\eta}{\tau}\right)y^k + \rho\left(\frac{\eta}{\tau}\bar{u}^k - \bar{w}^k\right).
\end{aligned}
\end{equation}Introduce $\sigma \coloneqq \gamma^2/\eta$ and use the identity $$\prox_{(\eta \tilde{g})^{*}}(u) = \eta B^\top \prox_{\sigma  g^*}(\sigma B u)$$ that holds for all $u \in \R^{n + m}$ (see \cite[Equation~(34)]{o2020equivalence}), to reformulate \eqref{eq::cp_alg_2} as
\begin{equation}\label{eq::cp_alg_3}
\begin{aligned}
    \bar{u}^k &= \left(\prox_{\tau f}(y^k_1), 0\right) \\
    \bar{w}^k &= \eta B^\top \prox_{\sigma g^*}\left(\sigma B\left(\left(1 + \frac{\eta}{\tau} \right)\bar{u}^k - \frac{\eta}{\tau}y^k\right)\right) \\
    y^{k+1} &= \left( 1 - \rho \frac{\eta}{\tau}\right)y^k + \rho\left(\frac{\eta}{\tau}\bar{u}^k - \bar{w}^k\right).
\end{aligned}
\end{equation} Introduce the variables $u^k,w^k\in \R^{n+m}$ for all $k\in\N$ such that
\begin{align*}
    u^{k+1} &= \left(1 - \rho \frac{\eta}{\tau}\right) u^k + \rho \frac{\eta}{\tau} \bar{u}^k \\
    w^{k} &= u^{k} - y^k
\end{align*} with $u^0= (u^0_1, 0)$ so that
\begin{equation}\label{eq::cp_variables}\begin{cases}
    y^0_1 &= u^0_1 - w^0_1\\
    y^0_2 &= - w^0_2
\end{cases}.\end{equation} Moreover,
\begin{align*}
    w^{k+1} &= u^{k+1} - y^{k+1} \\
    &= \left(1 - \rho \frac{\eta}{\tau}\right)(u^k - y^k) + \rho \bar{w}^k \\
    &= \left(1 - \rho \frac{\eta}{\tau}\right)w^k+ \rho \bar{w}^k 
\end{align*} and so we reformulate \eqref{eq::cp_alg_3} as
\begin{equation}\label{eq::cp_alg_4}
\begin{aligned}
    \bar{u}^k &= \left(\prox_{\tau f}(u^k_1 - w^k_1), 0\right) \\
    \bar{w}^k &= \eta B^\top \prox_{\sigma g^*}\left(\sigma \frac{\eta}{\tau} Bw^k + \sigma B\left(\left(1 + \frac{\eta}{\tau} \right)\bar{u}^k - \frac{\eta}{\tau}u^k\right)\right) \\
    u^{k+1} &= \left(1 - \rho \frac{\eta}{\tau}\right) u^k + \rho \frac{\eta}{\tau} \bar{u}^k \\
    w^{k+1} &= \left(1 - \rho \frac{\eta}{\tau}\right)w^k+ \rho \bar{w}^k.
\end{aligned} 
\end{equation}
Since $\bar{u}^k_2 = 0$ for all $k \in \N$ and $u^0_2 = 0$, we conclude that $u^k_2 = 0$ for all $k \in \N$. Let $$\bar{z}^k \coloneqq \prox_{\sigma g^*}\left(\sigma \frac{\eta}{\tau}Bw^k + \sigma B\left(\left(1 + \frac{\eta}{\tau} \right)\bar{u}^k - \frac{\eta}{\tau}u^k\right)\right), $$ so that $\bar{w}^k = \eta B^\top \bar{z}^k$ holds for all $k \in \N$. Moreover, 
\begin{align}
    B\bar{w}^k &= \eta B B^\top \bar{z}^k \notag\\
    &= \frac{\eta}{\gamma^2} \bar{z}^k \label{eq::barzk}.
\end{align} Therefore, by choosing $w^0$ to be in the range of $B^\top$ also $w^k$ will be in the range of $B^\top$ for all $k \in \N$, i.e., $w^k = \tau B^\top z^k$ for some $z^k \in \R^{m}$. Moreover, 
\begin{align}
    Bw^k &= \tau B B^\top z^k \notag\\
    &= \frac{\tau}{\gamma^2} z^k \label{eq::zk}.
\end{align}Let $x^k \coloneqq u^k_1$ and $\bar{x}^k \coloneqq \bar{u}^k_1$ for all $k \in \N$. Then,
\begin{align*}
    x^{k+1} &= \left(1 - \rho \frac{\eta}{\tau}\right) x^k + \rho \frac{\eta}{\tau} \bar{x}^k \\
    z^{k+1} &= \frac{\gamma^2}{\tau} B w^{k+1} \\
    &= \left(1 - \rho \frac{\eta}{\tau}\right)\frac{\gamma^2}{\tau} B w^k+ \rho \frac{\gamma^2}{\tau} B \bar{w}^k \\
    &= \left(1 - \rho \frac{\eta}{\tau}\right)z^k+ \rho \frac{\eta}{\tau} \bar{z}^k,
\end{align*} where we in the last equality used \eqref{eq::barzk} and \eqref{eq::zk}. Given initial points $x^0 \in \R^n$ and $z^0 \in \R^m$, let $w^0 = \tau B^\top z^0$ and $u^0 = (x^0, 0)$ and set $y^0 \in \R^n \times \R^m$ according to \eqref{eq::cp_variables}. Note that when $C$ is non-singular, i.e., when $\gamma \norm{A} < 1$, then this mapping between $(x^0, z^0)$ and $y^0$ is a bijection. 

Now let's perform a final change of variables. At this point, we have that \eqref{eq::cp_alg_4} becomes
\begin{equation}\label{eq::cp_alg_5}
\begin{aligned}
    \bar{x}^{k} &= \prox_{\tau f}\left(x^{k} - \tau A^\top z^{k} \right)\\
    \bar{z}^{k} &= \prox_{\sigma g^*}\left(z^{k} + \sigma A\left(\left(1 + \frac{\eta}{\tau}\right)\bar{x}^k - \frac{\eta}{\tau} x^{k} \right) \right)\\
    x^{k+1} &= x^{k} + \rho\frac{\eta}{\tau}\left(\bar{x}^k - x^{k} \right)\\
    z^{k+1} &= z^{k} + \rho\frac{\eta}{\tau}\left(\bar{z}^k - z^{k} \right)
\end{aligned}
\end{equation} with the bounds $(\tau, \eta, \sigma, \rho) \in \R^4_{++}$ such that $\rho \in (0, \min\{2, 2\tau/\eta\})$ and $\eta \sigma \norm{A}^2 \leq 1$. Let $\theta = \eta/\tau$ and $\bar{\rho} = \rho \theta$, then \eqref{eq::cp_alg_5} becomes
\begin{equation*}
\left\lfloor
\begin{aligned}
    \bar{x}^{k} &= \prox_{\tau f}\left(x^{k} - \tau A^\top z^{k} \right)\\
    \bar{z}^{k} &= \prox_{\sigma g^*}\left(z^{k} + \sigma A\left(\left(1 + \theta \right)\bar{x}^k - \theta x^{k} \right) \right)\\
    x^{k+1} &= x^{k} + \bar{\rho}\left(\bar{x}^k - x^{k} \right)\\
    z^{k+1} &= z^{k} + \bar{\rho}\left(\bar{z}^k - z^{k} \right)
\end{aligned} \right.
\end{equation*} with bounds $\bar{\rho} \in (0, \min\{2, 2\theta\})$ and $\tau\sigma\norm{A}^2 \leq 1/\theta$, which is exactly \cref{alg::extended_chambolle_pock}.

\subsection{Extending Parallel Splitting} \label{sec::parallel_splitting}
In this section, we extend the so-called Parallel Splitting algorithm, see \cite[Proposition~28.7]{BauschkeCombettes2017}, which is the Douglas--Rachford algorithm applied to a consensus problem formulation. Consider a family of functions $\{f_i\}_{i = 1}^n$ where $f_i \in \mathcal{F}$ for all $i \in \{1, \dots, n\}$ and the optimization problem
\begin{equation} \label{eq::optimization_problem_parallel_splitting}
    \minimize_{x \in \Hil} \sum_{i = 1}^n f_i(x).
\end{equation} Let $\Delta \coloneqq \{y \in \Hil^n \mid y_1 = y_2 = \cdots = y_n\}$ and let $f : \Hil^n \to \R \cup \{+\infty\}$ be given by $f(y) \coloneqq \sum_{i = 1}^n f_i(y_i)$ for each $y \in \Hil^n$. Consider the lifted consensus problem
\begin{equation} \label{eq::optimization_problem_parallel_splitting_lifted}
    \minimize_{y \in \Hil^n} f(y) + \iota_{\Delta}(y),
\end{equation} which is equivalent to \eqref{eq::optimization_problem_parallel_splitting}, with sufficient optimality condition
\begin{equation} \label{eq::optimization_problem_parallel_splitting_lifted_optimality}
    0 \in \partial f(y) + \partial \iota_{\Delta}(y).
\end{equation} 
 Then, $$\left(J_{\alpha \partial f}(y)\right)_i = \prox_{\alpha f_i}(y_i)\quad (\forall i \in \{1, \dots, n\}),$$ and $$\left(\prox_{\beta \iota_{\Delta}}(y)\right)_i = \frac{1}{n}\sum_{j = 1}^n y_j \quad (\forall i \in \{1, \dots, n\}),$$ holds for all $y \in \Hil^n$. Moreover, the projection $\prox_{\beta \iota_{\Delta}}$ is a linear mapping independent of $\beta$. We can therefore let $\gamma = \beta/\alpha$ range freely when \cref{alg::main} is applied to \eqref{eq::optimization_problem_parallel_splitting_lifted_optimality}, as long as $\theta < \min\{2, 2/\gamma\}$. We state \cref{alg::main} applied to \eqref{eq::optimization_problem_parallel_splitting_lifted_optimality} in \cref{alg::extended_parallel_splitting}. This generalizes the parallel splitting method in \cite[Proposition~28.7]{BauschkeCombettes2017} that is obtained by letting $\gamma=1$.
\cref{alg::extended_parallel_splitting} converges for these parameters if \cref{eq::optimization_problem_parallel_splitting_lifted_optimality} has a solution.
\begin{algorithm}[H]
	\caption{Extended Parallel Splitting}
	\begin{algorithmic}[1]
        \State \textbf{Input:} $f : \Hil^n \to \R \cup \{+\infty\}$ as in \cref{sec::parallel_splitting}, $y^0 \in \Hil^n$, \\
        \hspace{13mm}$(\alpha, \gamma, \theta) \in \R^3_{++}$ such that $\theta < \min\{2, 2/\gamma\}$.
		\For {$k=0,1,2\ldots$}
		    \State \textbf{Update: }
            \begin{equation*}
            \left\lfloor
            \begin{aligned}
            p^k &= \frac{1}{n}\sum_{j = 1}^n y^k_j \\
            x^{k}_i &= \prox_{\alpha  f_i}(y^k_i) \quad (\forall i \in \{1, \dots, n\})\\ 
            q^k &= \frac{1}{n}\sum_{j = 1}^n x^k_j \\
            y^{k+1}_i &= y^k_i + \theta \left(\left(1 + \gamma\right) q^k - \gamma p^k - x^k_i \right) \quad (\forall i \in \{1, \dots, n\})
            \end{aligned} \right.
            \end{equation*}
		\EndFor
	\end{algorithmic}
\label{alg::extended_parallel_splitting}
\end{algorithm}

\subsection{Extending Generalized Alternating Projections}

In this section, we will derive an extended version of the generalized alternating projections method \cite{GAP_Agmon,GAP_Motzkin,GAP_Bregman,falt-optimal, Dao_2018_Feasibility}, which is also called the averaged relaxed alternating projections method or generalized Douglas--Rachford splitting. 

Let $C\subset\Hil$ and $D\subset\Hil$ be nonempty, closed, and convex sets.
We consider the problem
\begin{equation}
    \label{eq::Feasibility}
        \text{find } x\in \Hil \text{ such that } x\in C\cap D.
\end{equation}
Using the notation of \cref{eq::reflected_res_rem}, the Generalized Alternating Projections method has the update
\begin{equation*}
   z^{k+1} = (1-\lambda)z^k + \lambda R_D^{\mu_2}R_C^{\mu_1}z^k
\end{equation*}
for $\lambda\in (0,1]$ and $(\mu_1, \mu_2)\in (0,2)$ or $\lambda\in (0,1)$ and $(\mu_1, \mu_2)\in(0,2]$, see for example \cite{falt-optimal}. In \cite{Dao_2018_Feasibility} convergence is shown for larger values of the relaxation parameter $\lambda$. 

By applying \cref{alg::main} to an appropriate reformulation, we get \cref{alg::Extended_GAP}. This extends generalized alternating projections which is obtained by letting $\gamma=\beta/\alpha=1$ and requires $\mu_1,\mu_2\in(0,2)$ and $\lambda\in(0,1)$. This is, to the best of the authors' knowledge, the first projection method of this type that allows for relaxed projection steps that are longer than reflections.

\begin{algorithm}[htbp]
	\caption{Extended Generalized Alternating Projections}
	\begin{algorithmic}[1]
	    \State \textbf{Input:} Nonempty closed convex sets $C, D\subset \Hil$, $z^0 \in \Hil$, and\\
        $\qquad\qquad(\mu_1, \mu_2, \lambda) \in (0, 1 + \gamma]\times (0, 1 + \gamma^{-1}]\times \left(0, \min\left\{\frac{2}{1+\gamma}, \frac{2}{1 + \gamma^{-1}}\right\}\right)$ for some $\gamma\in \R_{++}$. 
		\For {$k=0,1,2,\ldots$}
		    \State \textbf{Update: } \begin{equation*}
                \begin{aligned}
                   z^{k+1} = (1-\lambda)z^k + \lambda R_D^{\mu_2}R_C^{\mu_1}z^k
                \end{aligned} 
                \end{equation*}
		\EndFor
	\end{algorithmic}
    \label{alg::Extended_GAP}
\end{algorithm}

The generalized alternating projections method can be derived by applying Douglas--Rachford splitting to the problem
\begin{align} \label{eq::alt_proj_problem}
    \minimize_{x\in\Hil}\;c\dist_{C}^2(x)+d\dist_{D}^2(x),
\end{align}
with constants $c,d\in (0,+\infty]$
and the interpretation that $(+\infty)\cdot \dist_{C}^2 = \iota_{C}$. If $C\cap D$ is nonempty, this problem is equivalent to the feasibility problem \cref{eq::Feasibility}, and the associated inclusion problem is solvable under mild assumptions (see for example \cite{BauschkeCombettes2017}).

By applying \cref{alg::main} to \cref{eq::alt_proj_problem}, we instead get the Extended Generalized Alternating Projections algorithm stated in \cref{alg::Extended_GAP}, which we will now derive.

The proximal operator of the squared distance function, with step size $\alpha\in \R_{++}$, satisfies
\begin{align*}
    \prox_{\alpha c\dist_C^2} = \frac{c^{-1}}{c^{-1}+2\alpha}\Id+\frac{2\alpha}{c^{-1}+2\alpha}\Pi_C,
\end{align*}
and analogously for $\prox_{\beta d\dist_D^2}$ with the convention that $1/\infty = 0$. The corresponding relaxed proximal operator defined in \eqref{eq::reflected_res_rem} becomes
\begin{align*}
    R_{\alpha c\dist_C^2}^{\alpha,\beta} &= \left(1+\frac{\beta}{\alpha}\right)\prox_{\alpha c\dist_C^2} - \frac{\beta}{\alpha}\Id\\
    &=\left(1+\frac{\beta}{\alpha}\right)\left(\frac{c^{-1}}{c^{-1}+2\alpha}\Id+\frac{2\alpha}{c^{-1}+2\alpha}\Pi_C\right) - \frac{\beta}{\alpha}\Id\\
    &=\left(1-\frac{2(\alpha+\beta)}{c^{-1}+2\alpha}\right)\Id + \frac{2(\alpha+\beta)}{c^{-1}+2\alpha }\Pi_C.
\end{align*}
The expression $\frac{2(\alpha+\beta)}{c^{-1}+2\alpha }$ is increasing for $c>0$ and $\frac{2(\alpha+\beta)}{c^{-1}+2\alpha }\in\left(0,1+\frac{\beta}{\alpha}\right]$ for $c\in (0, +\infty]$. Let us introduce $\mu_1\coloneqq\frac{2(\alpha+\beta)}{c^{-1}+2\alpha}\in\left(0,1+\frac{\beta}{\alpha}\right]$ and $\mu_2\coloneqq\frac{2(\alpha+\beta)}{d^{-1}+2\beta}\in\left(0,1+\frac{\alpha}{\beta}\right]$, where $\mu_1$ and $\mu_2$ can take any values in these intervals for appropriate choices of $c$ and $d$ respectively, and define
\begin{align*}
    R_C^{\mu_1} &= (1-\mu_1)\Id+\mu_1\Pi_C\qquad\text{and}\qquad R_D^{\mu_2} = (1-\mu_2)\Id+\mu_2\Pi_D.
\end{align*}

Applying our Douglas--Rachford algorithm reformulation in \eqref{eq::alg_composition_rem} to this problem formulation gives the iteration
\begin{align*}
    z^{k+1} &= \left(1-\frac{\theta\beta}{\alpha+\beta}\right)z^k + \frac{\theta\beta}{\alpha+\beta}R_{D}^{\mu_2}R_{C}^{\mu_1}z^k
    = \left(1-\lambda\right)z^k + \lambda R_{D}^{\mu_2}R_{C}^{\mu_1}z^k,
\end{align*}
where $\lambda :=\frac{\theta\beta}{\alpha+\beta}$. Now, by \cref{thm::main_theorem} we have convergence if $\theta\in\left(0,\min\{2,2\alpha/\beta\}\right)$, implying that 
\begin{align*}
    \lambda \in \left(0,\min\left\{\frac{2\beta}{\alpha+\beta},\frac{2\alpha}{\alpha+\beta}\right\}\right).
\end{align*}
Let us also introduce $\gamma\coloneqq\beta/\alpha>0$, to get that \cref{alg::Extended_GAP} converges if
\begin{align*}
    \mu_1\in(0,1+\gamma]\qquad\text{and}\qquad \mu_2\in(0,1+\gamma^{-1}]\qquad\text{and}\qquad \lambda\in\left(0,\min\left\{\frac{2}{1+\gamma},\frac{2}{1+\gamma^{-1}}\right\}\right).
\end{align*}

\section{Conclusions}
\label{sec::conclusions}

In this work, we have shown that, in the convex optimization setting, the Douglas--Rachford splitting method is not the unique unconditionally convergent, frugal, no-lifting resolvent-splitting scheme. When the operators are subdifferentials of proper, closed, convex functions, we identify a strictly larger class of such methods and provide a complete, sharp characterization of all parameter choices that yield unconditional convergence. This characterization, in turn, immediately yields new, provably convergent families of ADMM- and Chambolle--Pock–type algorithms via their Douglas–Rachford reformulations, including parameter regimes that were previously not known to converge.

Our Lyapunov functions explicitly involve function values, placing the analysis beyond the standard averaged-operator framework. Because the Lyapunov construction uses function values of only the function whose subdifferential is applied first in each iteration, our results extend to mixed problems combining one subdifferential operator with one general maximally monotone operator, provided the algorithm respects this order. Preliminary numerical evidence from \cite{UpadhyayaTaylorBanertGiselsson2025AutoLyap} suggests that the same parameter region may remain valid when the order of the operators is reversed, although establishing this rigorously will require a new Lyapunov analysis. We formalize this in the following conjecture.

\begin{conjecture} \label{conj}
\cref{alg::main} converges unconditionally for $(\alpha, \beta, \theta) \in \R^2_{++} \times (\R \setminus \{0\})$ over $\mathcal{A} \times \partial\mathcal{F}$ if and only if $(\alpha, \beta, \theta) \in S^{(1)} \cup S^{(2)}$, i.e., $$S(\mathcal{A} \times \partial\mathcal{F}) = S(\partial\mathcal{F} \times \partial \mathcal{F}) = S^{(1)} \cup S^{(2)}.$$
\end{conjecture}

Beyond the Douglas--Rachford setting, we expect that many fixed-parameter methods for monotone inclusions may admit strictly larger regions of guaranteed convergence when their analysis is specialized to the convex optimization setting.

\newpage
\begin{appendices}
\section[]{Verification of Equality \eqref{eq::lem_equations} in \cref{lem::Lyapunov_def_alternative_2} and \eqref{eq::lemma_equality} in \cref{lemma::analysis}} \label{appendix}
The results shown in this Appendix are also verified symbolically in\footnote{\url{https://github.com/MaxNilsson-phd/Symbolic_verification_of_Extending_Douglas--Rachford_Splitting_for_Convex_Optimization}}. 

We will occasionally place $*$ in the lower-triangular part of a matrix to avoid repeating entries. The resulting matrix should always be interpreted as symmetric; for example, $$
\begin{bmatrix}
    a & b & c \\
    * & d & e \\
    * & * & f 
\end{bmatrix} = 
\begin{bmatrix}
    a & b & c \\
    b & d & e \\
    c & e & f 
\end{bmatrix}.$$
\subsection[]{Verification of \eqref{eq::lem_equations} in \cref{lem::Lyapunov_def_alternative_2}}
\label{subsec::Proof_of_Lemma_Lyapunov}
Let $\bar{V}_k$ denote the right hand side of the first equality of \eqref{eq::lem_equations}. Expanding the norms of the quantities in \cref{def::V_R} and \cref{lem::Lyapunov_def_alternative_2} gives the following expressions in terms of quadratic forms for all $k \in \N$,
    \begin{align*}
        V_k &= \mathcal{Q}\left(
        \underbrace{
        \begin{bmatrix}
         \theta - \tau + 1 & -\frac{\theta(1+\tau)}{2} & 0 & 0 &\tau - 1 - \frac{\theta}{2}\\[6pt]
        * & \theta \tau & 0 & 0 &\frac{\theta}{2}\\[6pt]
        * & * & 0 & 0 & 0\\[6pt]
        * & * & * & 0 & 0\\[6pt]
        * & * & * & * & 1
        \end{bmatrix}}_{=: Q_V^{(1)}} ,
        \begin{bmatrix}
            x^k_1 - x^\star \\[6pt]
            x^k_2 - x^\star \\[6pt]
            x^{k+1}_1 - x^\star \\[6pt]
            x^{k+1}_2 - x^\star \\[6pt]
            z^k - z^\star
        \end{bmatrix}
        \right) + 2\alpha(1-\tau) D_f(x^k_1, x^\star, u^\star), \\
        \bar{V}_k &= \mathcal{Q}\left(
        \underbrace{
        \begin{bmatrix}
         \theta + \tau - 1 & -\frac{\theta(1+\tau)}{2} &0 & 0 & - \frac{\theta}{2}\\[6pt]
        * & \theta \tau &0 & 0 &  \frac{\theta}{2}\\[6pt]
            * & * & 0 & 0 & 0\\[6pt]
        * & * & * & 0 & 0\\[6pt]
        * & * & * & * &1
        \end{bmatrix}}_{=: Q_V^{(2)}} ,
        \begin{bmatrix}
            x^k_1 - x^\star \\[6pt]
            x^k_2 - x^\star \\[6pt]
            x^{k+1}_1 - x^\star \\[6pt]
            x^{k+1}_2 - x^\star \\[6pt]
            z^k - z^\star
        \end{bmatrix}
        \right) + 2\alpha(\tau - 1) D_f(x^\star, x^k_1, u^k_1)\\
        &= \mathcal{Q}\left( Q_V^{(1)}, \begin{bmatrix}
            x^k_1 - x^\star \\[6pt]
            x^k_2 - x^\star \\[6pt]
            x^{k+1}_1 - x^\star \\[6pt]
            x^{k+1}_2 - x^\star \\[6pt]
            z^k - z^\star
        \end{bmatrix}\right) - 2(\tau-1)\langle x_1^k-x^\star,\underbrace{z^k-z^\star-( x_1^k-x^\star)}_{=\alpha(u_1^k-u^\star)}\rangle + 2\alpha(\tau - 1) D_f(x^\star, x^k_1, u^k_1).
        \end{align*}
        By \cref{def::algorithm_sequences} and by \cref{lem::Bregman_two_point} we then have that
        \begin{align*}
            \bar{V}_k-V_k = -2\alpha(\tau-1)\langle x_1^k -x^\star, u_1^k-u^\star\rangle + 2\alpha(\tau-1) \left(D_f(x_1^k, x^\star, u^\star) + D_f(x^\star, x_1^k, u_1^k)\right) = 0.
        \end{align*}

        Similarly, if we let $\bar{R}_k$ denote the right hand side of the second equality of \eqref{eq::lem_equations}, we obtain the following expressions in terms of quadratic forms for all $k \in \N$,
        \begin{align*}
        R_k &= \mathcal{Q}\left(
        \underbrace{
        \begin{bmatrix}
        \theta(2\tau-\theta) - \tau + 1 & -\theta(2\tau - \theta) & -(1 - \tau) & 0 & 0\\[6pt]
        * & \theta(2\tau - \theta) & 0 & 0 & 0\\[6pt]
        * & * & 1 - \tau & 0 & 0\\[6pt]
        * & * & * & 0 & 0\\[6pt]
        * & * & * & * & 0
        \end{bmatrix}}_{=: Q_R^{(1)}}
        ,
        \begin{bmatrix}
            x^k_1-x^\star \\[6pt]
            x^k_2-x^\star \\[6pt]
            x^{k+1}_1 -x^\star\\[6pt]
            x^{k+1}_2-x^\star\\[6pt]
            z^k-z^\star
        \end{bmatrix}
        \right) \\
        &\qquad+ 2\alpha(1-\tau)D_f(x_1^k, x_1^{k+1}, u_1^{k+1}),\\
        \bar{R}_k &= \mathcal{Q}\left(
        \underbrace{
        \begin{bmatrix}
         \theta(2 - \theta) + \tau - 1 & \theta(\theta-1-\tau) & (\theta-1)(\tau-1) & 0 & 0\\[6pt]
        * & \theta(2\tau - \theta) & -\theta(\tau-1) & 0 & 0 \\[6pt]
        * & * & \tau - 1 & 0 & 0\\[6pt]
        * & * & * & 0 & 0\\[6pt]
        * & * & * & * & 0
        \end{bmatrix}}_{=: Q_R^{(2)}}
        ,
        \begin{bmatrix}
            x^k_1-x^\star \\[6pt]
            x^k_2-x^\star \\[6pt]
            x^{k+1}_1-x^\star\\[6pt]
            x^{k+1}_2-x^\star\\[6pt]
            z^k-z^\star
        \end{bmatrix}
        \right)\\ &\qquad+ 2\alpha(\tau-1)D_f( x_1^{k+1}, x_1^k,u_1^{k}).
        \end{align*}
        Moreover
        \begin{align} \label{eq::Q_Rs}
            Q_R^{(2)} - Q_R^{(1)} &=  (\tau-1)\begin{bmatrix}
                 2-2\theta & \theta & \theta-2 & 0 & 0\\
                * & 0 & -\theta & 0 & 0\\
                * & * & 2 & 0 & 0\\
                * & * & * & 0 & 0\\
                * & * & * & * & 0
            \end{bmatrix}.
        \end{align}
    By \cref{def::algorithm_sequences} and by \cref{lem::Bregman_two_point} we have that the difference of the Bregman terms of $\bar{R}_k$ and $R_k$ is
    \begin{align*}
        &2\alpha(\tau-1)\left(D_f(x_1^{k+1}, x_1^k,u_1^{k}) + D_f(x_1^k, x_1^{k+1}, u_1^{k+1})\right) = \langle x_1^k-x_1^{k+1}, u_1^k - u_1^{k+1}\rangle\\
        &= 2(\tau-1)\langle x_1^k - x_1^{k+1}, z^k - x_1^k - (z^k + \theta(x_2^k-x_1^k)- x_1^{k+1})\rangle \\
        &= 2(1-\tau)\langle x_1^k-x_1^{k+1}, x_1^{k}-x_1^{k+1}+\theta(x_2^k-x_1^k) \rangle \\
        &= (1-\tau)\mathcal{Q}\left(\begin{bmatrix}2-2\theta & \theta & \theta-2\\[6pt]
        * & 0 & -\theta\\[6pt]
        * & * & 2
        \end{bmatrix}\begin{bmatrix}
            x_1^k-x^\star\\[6pt]
            x_2^k-x^\star\\[6pt]
            x_1^{k+1}-x^\star
        \end{bmatrix}\right).
    \end{align*}
    Together with \eqref{eq::Q_Rs} this shows that $\bar{R}_k -R_k = 0$.\\

    \subsection[]{Verification of Equality \eqref{eq::lemma_equality} in \cref{lemma::analysis}}
    \label{subsec::Proof_of_LE}
    What remains to show is the Lyapunov equality \cref{eq::lemma_equality}. We use the quadratic expressions from \cref{subsec::Proof_of_Lemma_Lyapunov}.
    We have for all $k\in \N$ that
    \begin{align*}
    V_{k+1} &= \mathcal{Q}\left(Q_V^{(1)}, \begin{bmatrix}
        x^{k+1}_1-x^\star \\
            x^{k+1}_2-x^\star \\
            x^{k+2}_1-x^\star\\
            x^{k+2}_2-x^\star\\
            z^{k+1}-z^\star
    \end{bmatrix}\right) + 2\alpha(1-\tau) D_f(x^{k+1}_1, x^\star, u^\star).\\
    \end{align*}
    Since the third and fourth rows and columns of $Q_V^{(1)}$ are zero, we get that the quadratic form in $V_{k+1}$ can be expressed as follows:
    {
    \allowdisplaybreaks
    \begin{align*}
    \mathcal{Q}\left(Q_V^{(1)}, \begin{bmatrix}
        x^{k+1}_1-x^\star \\
            x^{k+1}_2-x^\star \\
            x^{k+2}_1-x^\star\\
            x^{k+2}_2-x^\star\\
            z^{k+1}-z^\star
    \end{bmatrix}\right) &= \mathcal{Q}\left(Q_V^{(1)}, \begin{bmatrix}
        x^{k+1}_1-x^\star \\
            x^{k+1}_2-x^\star \\
            0\\
            0\\
            z^{k+1}-z^\star
    \end{bmatrix}\right) =
     \mathcal{Q}\left(Q_V^{(1)}, \underbrace{\begin{bmatrix}
        0 & 0 & 1 & 0 & 0\\
        0 & 0 & 0 & 1 & 0\\
        0 & 0 & 0 & 0 & 0\\
        0 & 0 & 0 & 0 & 0\\
        -\theta & \theta & 0 & 0 & 1\\
    \end{bmatrix}}_{=:M}\begin{bmatrix}
            x^k_1-x^\star \\
            x^k_2-x^\star \\
            x^{k+1}_1-x^\star\\
            x^{k+1}_2-x^\star\\
            z^k-z^\star
        \end{bmatrix}\right)\\&= \mathcal{Q}\left(M^\top\begin{bmatrix}
         \theta - \tau + 1 & -\frac{\theta(1+\tau)}{2} & 0 & 0 &\tau - 1 - \frac{\theta}{2}\\[6pt]
        * & \theta \tau & 0 & 0 &\frac{\theta}{2}\\[6pt]
        * & * & 0 & 0 & 0\\[6pt]
        * & * & * & 0 & 0\\[6pt]
        * & * & * & * & 1
        \end{bmatrix}M, \begin{bmatrix}
            x^k_1-x^\star \\
            x^k_2-x^\star \\
            x^{k+1}_1-x^\star\\
            x^{k+1}_2-x^\star\\
            z^k-z^\star
        \end{bmatrix}\right)\\
    &= \mathcal{Q}\left(\underbrace{\begin{bmatrix}
        \theta^2 & - \theta^2 & -\theta(\tau-1-\frac{\theta}{2}) & - \frac{\theta^2}{2} & -\theta\\[6pt]
        * &  \theta^2 & \theta(\tau-1-\frac{\theta}{2}) &  \frac{\theta^2}{2} & \theta\\[6pt]
        * & * & 1 - \tau + \theta& - \frac{\theta(1+\tau)}{2} & \tau - 1 - \frac{\theta}{2}\\[6pt]
        * & * & * & \theta \tau & \frac{\theta}{2}\\[6pt]
        * & * & * & * & 1
    \end{bmatrix}}_{=M^\top Q_V^{(1)} M}, \begin{bmatrix}
            x^k_1-x^\star \\[6pt]
            x^k_2-x^\star \\[6pt]
            x^{k+1}_1-x^\star\\[6pt]
            x^{k+1}_2-x^\star\\[6pt]
            z^k-z^\star
        \end{bmatrix}\right),
    \end{align*}
    }\relax
   where the first equality is due to the zero columns of $Q_V^{(1)}$ and the second equality follows from the update step of \cref{alg::main}. 
    
    For the Bregman-type distances we have by \cref{lem::Bregman_three_point} that 
        {
        \allowdisplaybreaks
        \begin{align*}
            D_f(x^k_1, x^\star, u^\star) &- D_f(x^{k+1}_1, x^\star, u^\star) - D_f(x^k_1, x^{k+1}_1, u^{k+1}_1) = \langle u^{k+1}_1 - u^\star, x^k_1 - x^{k+1}_1 \rangle \\
            &=  \frac{1}{\alpha}\langle z^{k+1} - z^\star - (x^{k+1}_1 - x^\star), x^k_1 - x^{k+1}_1 \rangle  \\
            &= \frac{1}{\alpha}\langle z^{k} - z^\star + \theta(x^k_2 - x^k_1) - (x^{k+1}_1 - x^\star), x^k_1 - x^{k+1}_1 \rangle  \\
            &= \frac{1}{\alpha}\mathcal{Q}\left(
            \underbrace{\begin{bmatrix}
                -\theta & \tfrac{\theta}{2} & \tfrac{\theta - 1}{2} & 0 & \tfrac{1}{2} \\[6pt]
                * & 0 & -\tfrac{\theta}{2} & 0 & 0 \\[6pt]
                * & * & 1 & 0 & -\tfrac{1}{2} \\[6pt]
                * & * & * & 0 & 0 \\[6pt]
                * & * & * & * & 0
            \end{bmatrix}}_{=:Q_D}
            ,
            \begin{bmatrix}
                x^k_1 - x^\star \\[6pt]
                x^k_2 - x^\star \\[6pt]
                x^{k+1}_1 - x^\star \\[6pt]
                x^{k+1}_2 - x^\star \\[6pt]
                z^k - z^\star
            \end{bmatrix}\right).
        \end{align*}}\relax
        Combining this equality with our expressions for $V_k$, $V_{k+1}$ and $R_k$ we get for each $k\in \N$ that
        \begin{align*}
             V_k - V_{k+1} - R_k &=  \mathcal{Q}\left(Q_V^{(1)}-M^\top Q_V^{(1)}M-Q_R^{(1)},
            \begin{bmatrix}
                x^k_1 - x^\star \\
                x^k_2 - x^\star \\
                x^{k+1}_1 - x^\star \\
                x^{k+1}_2 - x^\star \\
                z^k - z^\star
            \end{bmatrix}
            \right) \\
            &\quad\;+ 2\alpha(1 - \tau)\left(D_f(x_1^k, x^\star, u^\star)- D_f(x_1^{k+1}, x^\star, u^\star) - D_f(x^k_1, x^{k+1}_1, u^{k+1}_1)\right)\\
            &= \mathcal{Q}\left( Q_V^{(1)}-M^\top Q_V^{(1)}M-Q_R^{(1)} + 2(1-\tau)Q_D, \begin{bmatrix}
                x^k_1 - x^\star \\
                x^k_2 - x^\star \\
                x^{k+1}_1 - x^\star \\
                x^{k+1}_2 - x^\star \\
                z^k - z^\star
            \end{bmatrix}\right)
        \end{align*} 
  
    This gives a quadratic expression for $V_k-V_{k+1}-R_k$. Note also that the remaining term of \cref{eq::lemma_equality}, $\theta \alpha I_k$, can be expressed as
        \begin{align}
        \theta \alpha I_k &=\notag\\
        &\theta\alpha\left( \langle u_1^k-u^\star, x_1^k-x^\star \rangle + \langle u_2^k+u^\star, x_2^k-x^\star \rangle+ \langle u_1^{k+1}-u^\star, x_1^{k+1}-x^\star \rangle + \langle u_2^{k+1}+u^\star, x_2^{k+1}-x^\star \rangle\right)\notag\\
        &= \theta \left\langle z^k - z^\star - (x^k_1 - x^\star), x^k_1 - x^\star \right\rangle \notag\\
        &\quad\;+ \theta\left\langle z^{k} - z^\star + \theta(x^k_2 - x^k_1) - (x^{k+1}_1 - x^\star), x^{k+1}_1 - x^\star \right\rangle \notag\\
        &\quad\;+ \theta\left\langle x^{k}_1 - x^\star + \tau(x^{k}_1 - x^{k}_2) - (z^{k} - z^\star) , x^{k}_2 - x^\star \right\rangle \notag\\
        &\quad\;+ \theta\left\langle x^{k+1}_1 - x^\star + \tau(x^{k+1}_1 - x^{k+1}_2) - (z^{k} - z^\star) - \theta(x^k_2 - x^k_1), x^{k+1}_2 - x^\star \right\rangle\notag\\
        &= \theta\mathcal{Q}\left(\underbrace{\begin{bmatrix}
        -1 & \frac{1+\tau}{2} & -\frac{\theta}{2} & \frac{\theta}{2} & \frac{1}{2}\\[6pt]
        * & -\tau & \frac{\theta}{2} & -\frac{\theta}{2} & -\frac{1}{2}\\[6pt]
        * & * & -1 & \frac{1+\tau}{2} & \frac{1}{2}\\[6pt]
        * & * & * & -\tau & -\frac{1}{2}\\[6pt]
        * & * & * & * & 0
        \end{bmatrix}}_{=:Q_I}, \begin{bmatrix}
                x^k_1 - x^\star \\[6pt]
                x^k_2 - x^\star \\[6pt]
                x^{k+1}_1 - x^\star \\[6pt]
                x^{k+1}_2 - x^\star \\[6pt]
                z^k - z^\star
            \end{bmatrix} \right) \label{eq::Q_I_def}.
        \end{align}
        
         Lastly we verify that
        \begin{align*}
            Q_V^{(1)}-M^\top Q_V^{(1)}M-Q_R^{(1)} + 2(1-\tau)Q_D&= \theta Q_I,
        \end{align*} where $Q_I$ is given by \eqref{eq::Q_I_def}:
    {
    \allowdisplaybreaks
     \begin{align*}
            Q_V^{(1)} - M^\top Q_V^{(1)}M - Q_R^{(1)} + 2(1-\tau)Q_D &= \begin{bmatrix}
         \theta - \tau + 1 & -\frac{\theta(1+\tau)}{2} & 0 & 0 &\tau - 1 - \frac{\theta}{2}\\[6pt]
        * & \theta \tau & 0 & 0 &\frac{\theta}{2}\\[6pt]
        * & * & 0 & 0 & 0\\[6pt]
        * & * & * & 0 & 0\\[6pt]
        * & * & * & * & 1
        \end{bmatrix}\\
        &- \begin{bmatrix}
        \theta^2 & - \theta^2 & -\theta(\tau-1-\frac{\theta}{2}) & - \frac{\theta^2}{2} & -\theta\\[6pt]
        * &  \theta^2 & \theta(\tau-1-\frac{\theta}{2}) &  \frac{\theta^2}{2} & \theta\\[6pt]
        * & * & 1 - \tau + \theta& - \frac{\theta(1+\tau)}{2} & \tau - 1 - \frac{\theta}{2}\\[6pt]
        * & * & * & \theta \tau & \frac{\theta}{2}\\[6pt]
        * & * & * & * & 1
    \end{bmatrix}\\
    &-  \begin{bmatrix}
        \theta(2\tau-\theta) - \tau + 1 & -\theta(2\tau - \theta) & -(1 - \tau) & 0 & 0\\[6pt]
        * & \theta(2\tau - \theta) & 0 & 0 & 0\\[6pt]
        * & * & 1 - \tau & 0 & 0\\[6pt]
        * & * & * & 0 & 0\\[6pt]
        * & * & * & * & 0
        \end{bmatrix}\\
        &+ \begin{bmatrix}
                -\theta & \tfrac{\theta}{2} & \tfrac{\theta - 1}{2} & 0 & \tfrac{1}{2} \\[6pt]
                * & 0 & -\tfrac{\theta}{2} & 0 & 0 \\[6pt]
                * & * & 1 & 0 & -\tfrac{1}{2} \\[6pt]
                * & * & * & 0 & 0 \\[6pt]
                * & * & * & * & 0
            \end{bmatrix}2(1-\tau)\\
        &=
        \underbrace{\begin{bmatrix}
        -1 & \frac{1+\tau}{2} & -\frac{\theta}{2} & \frac{\theta}{2} & \frac{1}{2}\\[6pt]
        * & -\tau & \frac{\theta}{2} & -\frac{\theta}{2} & -\frac{1}{2}\\[6pt]
        * & * & -1 & \frac{1+\tau}{2} & \frac{1}{2}\\[6pt]
        * & * & * & -\tau & -\frac{1}{2}\\[6pt]
        * & * & * & * & 0
        \end{bmatrix}}_{=Q_I}\theta
        \end{align*}
    }
        which shows that $V_k-V_{k+1}-R_k =  \theta\alpha I_k$, so equality \cref{eq::lemma_equality} of \cref{lemma::analysis} holds.
\end{appendices}

\printbibliography

\end{document}